\documentclass[12pt]{article}
\usepackage{amsfonts, amsmath,amsthm}
\usepackage{comment}
\usepackage{xcolor}

\newtheorem{theorem}[equation]{Theorem}
\newtheorem{lemma}[equation]{Lemma}
\newtheorem{corollary}[equation]{Corollary}
\newtheorem{proposition}[equation]{Proposition}
\newtheorem{definition}[equation]{Definition}
\newtheorem{conjecture}[equation]{Conjecture}
\newtheorem{remark}[equation]{Remark}
\newtheorem{example}[equation]{Example}

\newtheorem{question}[equation]{Question}

\numberwithin{equation}{section}
\renewcommand{\thesection}{\arabic{section}}
\renewcommand{\theequation}{\thesection.\arabic{equation}}

\newcounter{subeq}
\numberwithin{subeq}{equation}
\renewcommand{\thesubeq}{\theequation.\arabic{subeq}}

\newenvironment{pf}
{\noindent\textbf{Proof: }}
{\mbox{}\hfill\textbf{ Q.E.D.}\medskip}

{\refstepcounter{subeq}\medskip\noindent(\thesubeq)\hfill}%
{\hfill\medskip}

\def\calI{\mathcal{I}}
\def\calO{\mathcal{O}}
\def\calP{\mathcal{P}}

\def\bbP{\mathbb{P}}
\def\bbR{\mathbb{R}}
\def\bbX{\mathbb{X}}

\def\mod{\text{mod}\;}

\begin{document}
\title{On $(i)$-Curves in Blowups of $\bbP^r$}


\author{
Olivia Dumitrescu \thanks{The first author was supported by the NSF-FRG grant DMS 2152130, NSF grant DMS1802082, and the Simons Foundation collaboration grant 855897.} \\
\small University of North Carolina at Chapel Hill \\ \small Chapel Hill, NC 27599-3250 \\
\and
Rick Miranda \\
\small Colorado State University \\ \small Fort Collins, CO 80523 USA}

\maketitle

\begin{abstract}
In this paper we study  $(i)$-curves with $i\in \{-1, 0, 1\}$ in the blown up projective space $\mathbb{P}^r$ in general points.
The notion of $(-1)$-curves was analyzed in the early days of mirror symmetry by Kontsevich
with the motivation of counting curves on a Calabi-Yau threefold.
In dimension two, Nagata studied planar $(-1)$-curves
in order to construct counterexample to Hilbert's 14th problem.

We introduce the notion of classes of $(0)$- and $(1)$-curves
in $\mathbb{P}^r$ with $s$ points blown up
and we prove that their number is finite if and only if the space is a Mori Dream Space.
We further introduce a bilinear form on a space of curves,
and a unique symmetric Weyl-invariant class, $F$,
(that we will refer to as the \emph{anticanonical curve class}).
For Mori Dream Spaces we prove that $(-1)$-curves can be defined arithmetically
by the linear and quadratic invariants determined by the bilinear form.
Moreover, $(0)$- and $(1)$-Weyl lines give the extremal rays
for the cone of movable curves in $\mathbb{P}^r$ with $r+3$ points blown up.
As an application, we use the technique of movable curves to reprove that
if $F^2\leq 0$ then $Y$ is not a Mori Dream Space
and we propose to apply this technique to other spaces.
\end{abstract}
\tableofcontents

\section{Introduction}
\subsection{Historical Background.}
The concept of \emph{ $(-1)$-curves} on a complex threefold
was introduced and studied by
Clemens \cite{C},
Friedman \cite[Section 8]{F}
and Kontsevich \cite[Section 1.4]{maxim1} and \cite[Section 2.3]{maxim2}
as a \emph{smooth rational curve with normal bundle isomorphic to $\calO(-1) \oplus \calO(-1)$}.
The interplay between mathematics and physics,
in the early days of Gromov-Witten theory,
and the role of $(-1)$-curves on a $3$-dimensional Calabi Yau 
is largely exposed in \cite{CK}.
The connection to enumerative geometry started from the influential paper \cite{Cd}
via the count of rational curves on a quintic threefold
by solving Picard-Fuchs equations on its mirror pair.
While developing the theory of mirror symmetry,
in \cite{maxim1} Kontsevich predicts that there are infinitely many $(-1)$-curves
on a Calabi-Yau threefold with prescribed numbers in each degree.

In this paper we study a natural generalization of this concept
to higher-dimensional varieties,
and to other normal bundles;
in particular we make the following definition:
\begin{definition}\label{i-curves}
Fix $i\in \{-1, 0, 1\}$.
An \emph{$(i)$-curve} on a smooth $r$-dimensional variety $X$
is a smooth irreducible rational curve
with normal bundle isomorphic to $\calO(i)^{\oplus(r-1)}$.
\end{definition}

We focus on the case of $Y^r_s$
which is the blowup of $\bbP^r$ at $s$ general points,
which are referred to as the \emph{base points}.
For this we are motivated by understanding the structure of the set of classes
of $(i)$-curves in the Chow ring,
in hopes of obtaining numerical criteria.
In the two-dimensional case,
this is closely related to the finite generation for the Cox ring of $Y^2_s$,
whose study was initiated by Nagata 
(\cite{nagatab} and \cite{nagata})
to understand the finite generation of rings of invariants and Hilbert's 14th problem. 
Nagata's counterexample directly used the infinity of $(-1)$-curves and their classes
to prove that for $s=16$ (later shown for $s \geq 9$) the corresponding Cox ring (isomorphic to the relevant ring of invariants) was not finitely generated,
see \cite[Theorem 2a]{nagata1}; 
his pioneering work contributed to the development of birational geometry.
Nagata's correspondence between {\it planar $(-1)$-curves and $(-1)$-Weyl lines} plays a key role there, and also in this work.

This theme of research continued in Mori's work on establishing
the \emph{Minimal Model Program},
and led to the identification of \emph{Mori Dream Spaces}:
those for which the Cox ring is finitely generated.

\subsection{Main Results.}



This paper is organized as follows.

In Section \ref{ChowRingY}
we present the main elements needed
from the Chow ring $A^*(Y^r_s)$ and its intersection form $(-,-)$.
In Section \ref{standardCremona} we give a straightforward description
of the action of the standard Cremona transformations
on curve classes $A^{r-1}(Y^r_s)$
in Proposition \eqref{CremonaAY}.
Finally in Section \ref{WeylGroupSubsection}
we present the Coxeter group theory that results in the basic
analysis of the Weyl group.

This leads us to define the bilinear form 
$\langle -,-\rangle$ on the curve classes $A^{r-1}(Y^r_s)$
and an anticanonical curve class $F \in A^{r-1}$
(meant to be dual in some sense to the anticanonical class $-K_Y \in A^1$)
which plays a central role in the analysis.
We denote the bilinear form induced by the Coxeter group theory on $A^1(Y^r_s)$
by $\langle -,-\rangle_1$
(known as the Dolgachev-Mukai pairing) and we exploit both forms systematically.
(For more details, see \cite{DM2}.)

In Section \ref{gen} we study general properties of $(i)$-curves
on an arbitrary smooth variety of dimension $r$,
before returning to focus on the $Y^r_s$ case.
We introduce the concept of \emph{$(i)$-Weyl lines}
and their classes to denote curves and classes
in the Weyl orbit of the proper transform 
of a line in $\bbP^r$ through $1-i$ of the base points.
We make corresponding definitions of $(i)$-Weyl hyperplanes
and observe numerical properties for their classes as well;
Corollary \eqref{div} of Section \ref{SectionBasicsCurves} proves that
$(i)$-Weyl hyperplanes and $(i)$-divisorial classes are equivalent in $\bbP^r$;
this remark extends Nagata's correspondence for divisors
from $i=-1$ \cite{DP} to $i=1$.
Some care must be taken with respect to
an assumption of irreducibility for the curves in question;
Example \eqref{exp} emphasizes this in the planar case.

In Section \ref{-1curvesMDS} we study numerical conditions that
provide a useful tool to prove the finiteness of $(i)$-curves.
In Section \ref{infinite}, we use the computation of Weyl actions on curves
to prove that the space $Y^r_s$ is a Mori Dream space if and only if 
it has infinitely many classes of $(1)$-curves.
Moreover, there are infinitely many $(-1)$-curves
if the number of points is at least $r+5$.
We conclude with applications of the theory of rigid and movable curves.
More precisely, we identify extremal rays for the movable cone of curves,
by exploiting the faces of the effective cone of divisors for Mori Dream Spaces in Section \ref{apl}. 

We close this section with a summary of what we believe are the most important results of the article.

The classification of the Mori Dream Spaces among the spaces $Y^r_s$
is conveniently expressed in terms of the bilinear forms $\langle -,-\rangle_1$ on $A^1$
and $\langle -,-\rangle$ on $A^{r-1}$
coming from the Coxeter group approach.
The existence of the bilinear form $\langle -,-\rangle$ on the curve classes
gives us a linear invariant $\langle c,F\rangle$
(which equals $(c \cdot -K_Y)$)
and a quadratic invariant $\langle c,c\rangle$
for curve classes $c$.
Although the curve class cannot detect the decomposition of the normal bundle
of a smooth rational curve, the linear invariant does detect its anticanonical degree,
and we say that a class $c$ is a \emph{numerical $(i)$-class}
if $\langle c,F\rangle =  2+i(r-1)$,
which is the value for the class of an $(i)$-curve.
This linear invariant is related to the virtual dimension
(namely $\chi$ of the normal bundle) of
a curve $C$ in $Y^r_s$, see (\ref{param}).

\begin{theorem} If $Y=Y^r_s$, then the following statements are equivalent:
\begin{enumerate}
\item $Y$ is a Mori Dream space.
\item The Coxeter group (and the Weyl group) is finite.
\item $F^2:=\langle F, F \rangle >0$
(which is identical to $\langle K_Y, K_Y \rangle_1 >0$ and also $(-K_Y \cdot F) > 0$).
(This is equivalent to $r=2, s \leq 8; \text{ or }r=3 \text{ or }4, s \leq r+4; \text{ or } r\geq 5, s \leq r+3$.)
\item $Y$ has finitely many $(0)$-curves (or equivalently, $(0)$-Weyl lines or $(0)$-numerical classes).
\item $Y$ has finitely many $(1)$-curves (or equivalently, $(1)$-Weyl lines or $(1)$-numerical classes).
\end{enumerate}
\end{theorem}


Corollary \ref{mds} implies the equivalence of the first three statements,
while Theorem \ref{finiteclass} and Corollary \ref{infinite01rgeq5}
implies part $(4)$ and $(5)$.

We conjecture (Conjecture \ref{HardConjecture}) that 
every $(i)$-Weyl line is an $(i)$-curve,
and it's easy to see that the class of an $(i)$-curve is a numerical $(i)$-class.
It is natural to ask: when are these notions identical?
They are not equivalent in arbitrary dimension.
We do prove Conjecture \ref{HardConjecture}
in certain cases (for $i = 0, 1$ and for $i=-1$ if $s \leq r+4$).

Example \eqref{F class in 3} gives a $(-1)$-curve
that is not a $(-1)$-Weyl line in $Y^3_8$
(the first case when the space $Y$ is not a Mori Dream Space).
The numerical $(-1)$-classes can also contain classes of curves of different genera.
Indeed, the same example of the $F$ class in $Y^3_8$
is the class of four $(-1)$-curves (four disjoint lines each through two of the $8$ points),
but also the class of an elliptic curve that is a complete intersection
of two quadrics in $\bbP^3$.
It is however also the class of a $(-1)$-curve
(indeed, four disjoint ones,
none of which are $(-1)$-Weyl lines).
It is therefore remarkable that for Mori Dream Spaces,
$(-1)$-numerical classes represent just one curve in $Y^r_s$,
which is a $(-1)$ curve.

The main results of Section \ref{-1curvesMDS} involve characterizing the $(i)$-curves
and $(i)$-Weyl lines using the numerical invariants; we have the following when $i=-1$.

\begin{theorem}\label{class}
Assume $r \geq 3$.  
Suppose that $C$ is a curve in $Y$ with class $c\in A^{r-1}(Y)$.
If $Y$ is a Mori Dream Space or $Y=Y^5_9$,
then the following are equivalent:
\begin{enumerate}
\item $C$ is an $(-1)$-curve.
\item $C$ is a $(-1)$-Weyl line.
\item $\langle c, c \rangle=3-2r$ and $\langle c, F \rangle = 3-r$.
\item $c = (1;1^2)$ or $c=(r,1^{r+3})$.
\end{enumerate}
\end{theorem}

To prove this result we study the notion of numerical $(-1)$-classes
that in dimension at least $3$ for the Mori Dream Space cases 
are each the class of a unique $(-1)$-curve.
In particular we prove that in the Mori Dream Space cases there is a finite number of such curves.
Theorem \ref{class} extends to irreducible $(0)$-curves with two exceptions:
the $F$ class if $r=3$ and $s=7$, and the class $2F$ if $r=4$ and $s=8$ (Remark \ref{rem}).
 Theorem \ref{class3} discusses $(1)$-curves in even dimensional spaces $Y^r_{r+3}$. Moreover, we note in particular that $Y^5_9$ is not a MDS and has an infinite Weyl group
but has finitely many $(-1)$-curves, which may be surprising. In Section \ref{infinite} we prove that
\[
\text{\emph{ If $s\geq r+5$ there are infinitely many $(-1)$-Weyl lines
(and hence $(-1)$-curves) on $Y$.}}
\]

Moreover, in Section \ref{appMoricone} if $\mathcal{Z}_{\geq 0}\langle -1 \rangle$ denotes the cone generated by classes of  $(-1)$-curves and classes of curves that meet all $(0)$-divisorial classes non-negatively:
\[
\text{\emph{The cone of classes of  effective curves in $Y^r_s$ is a subcone of $\mathcal{Z}_{\geq 0}\langle -1 \rangle$}}.
\]

Finally, Section \ref{apl} presents applications of the theory of $(i)$-curves
to the effective cone of divisors $\text{Eff}_{\bbR}(Y)$ on $Y=Y^r_{r+3}$.
The geometry of Mori Dream spaces $Y$ was previously analyzed via work of Mukai
and techniques from birational geometry of moduli spaces. We recall that in $Y^r_{r+3}$ the cone of effective divisors is closed, and
 a line bundle $L$ on a projective variety is pseudo-effective (effective) if and only if $L\cdot C\geq 0$ for all irreducible curves $C$ that move in a family covering $X$ \cite{BDPP}. Theorems \ref{effective cone} and  Theorem \ref{nms} imply

\begin{theorem}\label{extremal rays}
\begin{enumerate}
\item The extremal rays of the cone of movable curves in $Y^r_{r+3}$  are $(0)$-Weyl lines and $(1)$-Weyl lines.
\item $(0)$-Weyl lines and $(1)$-Weyl lines are extremal rays for the cone of movable curves in $Y^r_{s}$ for arbitrary $s$.
\end{enumerate}
\end{theorem}

As a corollary to part (2), the infinity of $(0)$- and $(1)$-Weyl lines give a different approach,
via the theory of movable curves,
 of the following result of Mukai (originally proved via the theory of divisors):
 \emph{If $F^2\leq 0$, then $Y^r_s$ is not a Mori Dream Space.}
 We leave as an open question to investigate if the theory of movable curves can be further applied to prove similar properties for other spaces.

\subsection*{Acknowledgements} 
We are indebted to C. Ciliberto for exposing the Example \ref{counterex} to us. The first author is a member of the Simion Stoilow Institute of Mathematics, Romanian Academy, 010702 Bucharest, Romania.

\section{The Chow ring, Cremona transformations, and the Weyl Group}\label{ChowRingY}
Let us consider the rational variety $Y=Y_s^r$
defined as the blowup of $\bbP^r$ at $s$ general points $p_1,\ldots,p_s$,
with blowup map $\pi:Y \to \bbP^r$.
Theorem \ref{CremonaAY}  contains
one of the main results of this paper:
\emph{the computations of the Weyl group orbits
for curves} directly on Chow ring of $Y$
(on which the Weyl group acts naturally)
without performing a sequence of flops.
Proposition \ref{antican} introduces the concept of the \emph{anti-canonical curve class $F$} that we will use throughout the paper.

There are two complications in using the Chow ring classes to study $(i)$-curves.
First, there is no numerical criterion for the rationality of a curve when $r \geq 3$.
(In the planar case, there is the genus formula 
expressing rationality in terms of the normal bundle and anticanonical degree,
emphasized in Proposition \ref{planar};
this was also exploited in \cite{DP} and in Section \ref{diva} in order 
to define divisorial $(i)$-classes.)

Second, even if one knows that a given class is represented by a smooth rational curve,
the normal bundle summands are difficult to compute in arbitrary dimension.
In particular, for a given degree of the normal bundle of a rational curve
it is not easy to describe simple sufficient conditions to make the normal bundle balanced.
However, we expect that for rational curves through general points,
the normal bundle should be as balanced as possible.

\begin{example}\label{line} For $i\in \{-1,0,1\}$ we have the following examples
of $(i)$-curves on $Y$:
\begin{enumerate}
\item The proper transform of a line through $1-i$ points is an $(i)$-curve, if $s \geq 1-i$.
\item The proper transform of the (unique) rational normal curve of degree $r$
through $r+2-i$ of the points is an $(i)$-curve, if $s \geq r+2-i$.
\end{enumerate}
\end{example}

These two examples are immediate,
given that the normal bundle of the line is $\calO(1)^{\oplus(r-1)}$,
the normal bundle of the rational normal curve is $\calO(r+2)^{\oplus(r-1)}$,
and upon blowing up a point the normal bundle is twisted by $-1$.

We make the following
\begin{conjecture}\label{HardConjecture}
A $(i)$-Weyl line
(a curve in the orbit of a line through $1-i$ points)
is an $(i)$-curve on $Y$.
\end{conjecture}

We note that the conjecture is trivially true in the planar case $r=2$;
in that case the Cremona transformation is an isomorphism.
We are not able to prove this in general, but we have results in special cases,
and in particular it is true for $i =0, 1$ or with $i=-1$ and $s \leq r+4$;
see Propositions \ref{s leq r+4 (-1)-curves} and \ref{ConjectureProofTheorem}.

In order to generate more examples of $(i)$-curves Theorem \ref{CremonaAY} gives the Weyl group orbit action on curves.

Weyl group actions on curves were considered for $r=3$ in \cite{LU} and $r=4$ in \cite{DM1}.
This action was used to study closure of the diminished base locus of divisors \cite{L}.

\subsection{The Chow ring of $Y^r_s$}

Let $H$ be the hyperplane class in $\bbP^r$,
and $E_i$, $1 \leq i \leq s$, be the exceptional divisors in $Y$.

The generators of the Chow ring $A^*(Y)$ of $Y$ are easy to describe;
in codimension zero we have only the class of the entire variety,
and in codimension $r$ we have only the point class.
In each of the intermediate codimensions $j$
we have the pullback (via $\pi$) of the general linear space in $\bbP^r$ of codimension $j$
(and hence of dimension $r-j$),
and, for each $i = 1,\ldots,s$,
the class of the general linear space of dimension $r-j$ inside the exceptional divisor $E_i$.

We will only require the divisor classes (i.e. $A^1(Y)$)
and the curve classes (i.e. $A^{r-1}(Y)$) in the rest of this paper.
We have already introduced notation above for the generators of $A^1$;
we will use $h$ to denote the class of the pullback of a general line in $\bbP^r$
and $e_i$ will be the class of the general line inside $E_i$ for each $i$.
The following is standard.

\begin{proposition}\label{antican}
With the above notation, we have:
\begin{itemize}
\item[(a)] The Chow group $A^0(Y)$ is one-dimensional, generated by the identity class $[Y]$.
\item[(b)] The Chow group $A^1(Y)$ has dimension $s+1$, generated by $H$ and $E_i$, $1\leq i \leq s$.
\item[(c)] The Chow group $A^{r-1}(Y)$ has dimension $s+1$, generated by $h$ and $e_i$, $1\leq i \leq s$.
\item[(d)] The Chow group $A^r(Y)$ is one-dimensional, generated by the class $[p]$ of a point $p$.
\item[(f)] Multiplication in $A(Y)$ is induced from the intersection form on $Y$,
and we have that the pairing $(-\cdot -): A^1 \times A^{r-1} \to A^r$ is given by
\[
(H \cdot h) = 1; \;\; (H\cdot e_i) = 0; \;\; (E_i \cdot h) = 0; \;\; (E_i \cdot e_j) = -\delta_{ij}
\]
where we have abbreviated the multiples of the point class $n[p]$ simply by the integer $n$.
\item[(g)] The canonical class of $Y$ in $A^1(Y)$ is given by
\[
K_Y = -(r+1) H + (r-1) \sum_{i=1}^s E_i
\]
and we define the \emph{anti-canonical curve class} $F_Y \in A^{r-1}(Y)$ to be
\[
F_Y = (r+1) h - \sum_{i=1}^s e_i.
\]
\end{itemize}
\end{proposition}

If $\bar{C} \subset \bbP^r$ is a curve of degree $d$
having multiplicity $m_i$ at $p_i$ for each $i$,
and $C \subset Y$ is the proper transform in $Y$,
then the class $[C] \in A^{r-1}(Y)$ is
$
[C] = d h - \sum_i m_i e_{i}
$
which can be easily deduced by intersecting with the basis elements of $A^1(Y)$.
Similarly for divisors, if $D$ is a divisor on $\bbP^r$
having multiplicity $m_i$ at $p_i$ for each $i$,
then the proper transform of $D$ in $Y$ has the class $d H - \sum_i m_i E_i$.
With this notation, the intersection pairing between divisors and curves can be written as
\[
(d H - \sum_i m_i E_i \cdot d' h - \sum_i m_i' e_i) = dd' - \sum_i m_i m_i'.
\]

\subsection{The standard Cremona transformations}\label{standardCremona}

The theory in the planar case is well understood
and beyond the planar case,
Weyl orbits of curves for $r=3$ were computed by \cite{LU},
and by the two authors for $r=4$ in \cite{DM1}.
The method used in these papers is more difficult
as it involves tracing these curves through a sequence of flops of lines and planes,
and the difficulty increases significantly for arbitrary dimension $r$.
We will exploit the standard Cremona transformation of $\bbP^r$ (centered at $r+1$ points) which we describe below.
The most important result of this section is Proposition \ref{CremonaAY} part (a) and (c) 
that generalizes formulas of \cite{LU} and \cite{DM1} from dimension $3$ and $4$ 
to arbitrary dimension.

The standard Cremona transformation of $\bbP^r$
(inverting the coordinates,
i.e., sending $[x_0:\cdots:x_r]$ to $[x_0^{-1}:\cdots:x_r^{-1}]$)
is realized geometrically by blowing up the $r+1$ coordinate points,
then the proper transforms of all coordinate lines, 
then the proper transforms of all coordinate $2$-planes, etc., until one blows up 
the proper transforms of all coordinate $(r-2)$-planes;
then one blows down the exceptional divisors starting with those over the coordinate lines,
then those over the coordinate $2$-planes, etc.,
finally ending by contracting the proper transforms of the coordinate hyperplanes.
We will index the $r+1$ coordinate points used here by $\{1,\ldots,r+1\}$,
considering them as the first $r+1$ of the $s$ points to be blown up to obtain $Y^r_s$.

After the first stage of performing the blowups, we arrive at a $r$-fold $\bbX^r_{r-1}$,
which has divisor classes
$H$ (the pullback of the hyperplane class on $\bbP^r$)
and $E_J$ for index sets $J \subset \{1,\ldots,r+1\}$ with $1 \leq |J| \leq r-1$
where $E_J$ is obtained by blowing up the proper transform of the span of the coordinate points indexed by $J$ under the blow up of smaller dimensional linear spans of $J$.
The linear system that defines the Cremona transformation on $\bbX^r_{r-1}$
is given by the Cremona transformation of the hyperplane class,
which is
\[
H' = rH - \sum_{i=1}^{r-1} (r-i)\sum_{J; |J|=i} E_J
\]
For $2 \leq |J| \leq r-1$, we have that the Cremona image of $E_J$ is given by
\[
E_J' = E_{J'}
\]
where $J'$ is the complement of $J$ in $\{1, \ldots, r+1\}$.
Each $E_i$ is transformed to the coordinate hyperplane through all other indices, and so
\[
E_i' = H - \sum_{J: |J|\leq r-1, i \notin J} E_J
\]
The subspace in the codimension one part of the Chow ring of $\bbX^r_{r-1}$
generated by the divisor classes $E_J$ with $2 \leq |J| \leq r-1$
(i.e., those exceptional divisors over the positive-dimensional linear coordinate spaces of $\bbP^r$)
is invariant under the Cremona transformation,
and the quotient space inherits the action.
This quotient space is naturally isomorphic to $A^1(Y^r_{r+1})$,
and the action extends (trivially) to an action on $A^1(Y^r_s)$ for any $s \geq r+1$.
The formulas above imply that, in the Chow ring of $Y^r_s$, we have
\begin{equation}\label{H'E'}
H' = rH-(r-1)\sum_{j=1}^{r+1} E_j \;\;\text{and}\;\;
E_i' = H - \sum_{j=1,j\neq i}^{r+1} E_j\;\;\text{if}\;\; i \leq r+1
\end{equation}
while $E_i' = E_i$ if $i > r+1$.

Since the $s$ points are general, any set of $r+1$ of them
can be the base points of a corresponding Cremona transformation.
For any subset $I$ of $r+1$ indices,
we will denote by $\phi_I$ the corresponding Cremona transformation,
which induces an action on the codimension one Chow space $A^1(Y^r_s)$.
The formulas given in (\ref{H'E'}) describe $\phi_{\{1,2,\ldots,r+1\}}$ acting on $A^1(Y^r_{r+1})$.

Dually, the subspace of the curve classes in the Chow ring of $\bbX^r_{r+1}$
spanned by the general line class and the general line classes inside of each $E_i$
is also invariant under the Cremona action.
This subspace is naturally isomorphic to $A^{r-1}(Y^r_{r+1})$,
and therefore we have a Cremona action there, which extends to $A^{r-1}(Y^r_s)$.

In the next Proposition we describe these actions explicitly, and leave the details of checking the formulas to the reader.

\begin{proposition}\label{CremonaAY}
Fix any $(r+1)$-subset $I \subset \{1,2,\ldots, s\}$.
\begin{itemize}
\item[(a)]
The action of $\phi_I$ on $A^{r-1}(Y^r_s)$ is given by sending $dh-\sum_i m_i e_i$ to $d'h-\sum_i m_i' e_i$ where
\begin{align*}
d' &= rd-(r-1)\sum_{i\in I} m_i =d+(r-1)t_{r-1} \\
m_i' &= d - \sum_{j\in I, j \neq i} m_j = m_i + t_{r-1} \;\;\text{for}\;\; i \in I \\
m_i' &= m_i \;\;\text{for}\;\; i \notin I
\end{align*}
for $t_{r-1} = d - \sum_{i \in I} m_i$.  It has order two.
In particular
$\phi_I(h) = rh-\sum_{i\in I} e_i$,
$\phi_I(e_j) = (r-1)h-\sum_{i\in I, i\neq j} e_j$ if $j\in I$,
and $\phi_I(e_j) = e_j$ if $j \notin I$.

\item[(b)]
The action of $\phi_I$ on $A^1(Y^r_s)$ is given by sending $dH-\sum_i m_i E_i$ to $d'H-\sum_i m_i' E_i$ where
\begin{align*}
d' &= rd - \sum_{j\in I}m_j = d+t_1 \\
m_i' &= (r-1)d - \sum_{j\in I, j\neq i} m_j = m_i + t_1 \;\;\text{for}\;\; i \in I \\
m_i' &= m_i \;\;\text{for}\;\; i \notin I
\end{align*}
for $t_1 = (r-1)d - \sum_{i \in I} m_i$.  It has order two.
In particular
$\phi_I(H) = rH-(r-1)\sum_{i\in I} E_i$,
$\phi_I(E_j) = H-\sum_{i\in I, i\neq j} E_j$ if $j\in I$,
and $\phi_I(E_j) = E_j$ if $j \notin I$.

\item[(c)]
The intersection pairing between $A^1(Y)$ and $A^{r-1}(Y)$ is $\phi_I$-invariant, i.e.,
for any class $D \in A^1(Y)$ and $C \in A^{r-1}(Y)$, we have
\[
(D \cdot C) = (\phi_I(D) \cdot \phi_I(C)).
\]
\end{itemize}
\end{proposition}

We will abbreviate $\phi = \phi_I$ for $I = \{1,2,\ldots,r+1\}$.

\subsection{The Weyl Group}\label{WeylGroupSubsection}
In this Section equation \ref{bil} introduces a bilinear form on $A^{r-1}$ (and on $A^1$) 
and Corollary \ref{mds} exploits properties of these forms in the language of Mori Dream Spaces.

Note that the symmetric group on the indices of the $s$ points
acts on all these spaces, and preserves the intersection form pairing also;
moreover if $\sigma$ is the permutation taking the subset $I$ to the subset $J$,
then $\sigma \phi_I \sigma^{-1} = \phi_J$.
The group $W$ generated by the Cremona transformations and the symmetric group
is called the \emph{Weyl Group} of these Chow spaces.

The canonical class $K_Y = -(r+1)H + (r-1)\sum_i E_i$ is the only symmetric $\phi$-invariant class
(up to scalars) in $A^1$ (and it is invariant under all $\phi_I$ therefore).
Dually, the anticanonical curve class $F =(r+1)h - \sum_{i=1}^s e_i $
is the only symmetric $\phi$-invariant class (up to scalars) in $A^{r-1}(Y)$,
and again it is invariant under all $\phi_I$.

The symmetric group is generated by the transpositions $\sigma_i = (i,i+1)$
for $1\leq i \leq s-1$,
and so the Weyl group generated by the Cremona transformation and the symmetric group
is generated by the elements $\phi$ and $\sigma_i$.
Each of these elements have order two.
Moreover it is easy to see that $\sigma_i$ and $\sigma_j$ commute
if and only if $|i-j|>1$,
and $\phi$ and $\sigma_i$ commute for all $i \neq r+1$.
Finally we have that $(\sigma_i\sigma_{i+1})^3 = 1$ for each $i$,
and $(\phi\sigma_{r+1})^3 = 1$.

We therefore see that the Weyl group actions on $A^1$ and on $A^{r-1}$
give representations of the Coxeter group associated to the $T_{2,r+1,s-r+1}$ graph:
\begin{equation}\label{Tpqrgraph}
\begin{matrix}
(1)  -  (2)  -  \cdots  -  (r)  - &  (r+1) & -  (r+2)  \cdots  -  (s-1) \\
     &     |           &    \\
     &     (0)       &     \\
\end{matrix}
\end{equation}
where the $(0)$ vertex corresponds to $\phi$,
and the $(i)$ vertex corresponds to $\sigma_i$ for $i \geq 1$.

We will describe this a bit further in the rest of this section,
but this is relatively well known; see
\cite{dolgachev83}, \cite{dolgort88}, \cite{dolgachev08} for example.
The relevant theory from Coxeter groups can be found in
\cite{Bo}, \cite{Hu}, and/or \cite{BB}.

There are $W$-invariant quadratic forms on $A^1$ and $A^{r-1}$,
defined by
\begin{equation}
q_1(dH-\sum_i m_iE_i) = (r-1)d^2-\sum_i m_i^2\;\;\text{and}\;\;
q_{r-1}(dh-\sum m_ie_i) = d^2 - (r-1)\sum_i m_i^2;
\end{equation}
we leave it to the reader to check that these are $\phi$-invariant
(they are clearly symmetric).
These give rise to associated bilinear forms: for $\alpha = 1$ or $\alpha = r-1$ we have
\begin{equation}\label{bil}
\langle x, y \rangle_\alpha = \frac{1}{2}(q_\alpha(x+y)-q_\alpha(x)-q_\alpha(y))
\end{equation}
and as usual $q_\alpha(x) = \langle x, x \rangle_\alpha$ for all $x$ and both values of $\alpha$.

We will usually abbreviate the bilinear form on the curve classes
as simply $\langle -,- \rangle$ without the subscript.

In the divisor case $\langle x, y \rangle_1$, is the Dolgachev-Mukai pairing
that has $\langle H, H \rangle_1 = r-1$, $\langle E_i, E_i \rangle_1 = -1$,
and all other values on the given basis elements for $A^1$ equal to zero.
In the curve case, when $r=3$ this quadratic invariant was observed and used in $\cite{LU}$;
its formula is
\[
\langle dh-\sum_i m_ie_i, d'h-\sum_i m_i'e_i\rangle = dd'-(r-1)\sum_i m_i m_i'.
\]

\begin{proposition} With the above notation, we have the following:
\begin{enumerate}
\item The action of $W$ on $A^{r-1}$ induces an action on
$K^\perp = \{c\in A^{r-1}\;|\; (c\cdot K) = 0\}$ which has dimension $s$.
\item For $c \in A^{r-1}$ we have $(-K \cdot c) = \langle F, c\rangle$.
Hence the orthogonal space to $F$
(with respect to the pairing given by the bilinear form $\langle -,-\rangle$)
is equal to $K^\perp$.
\item The action of $W$ on $K^\perp$
is isomorphic to the standard geometric representation of the Coxeter group.
\item The standard bilinear form on the standard geometric representation corresponds,
under this isomorphism,
to $(-1/2)$ times the restriction of the $\langle -,-\rangle$ pairing on $A^{r-1}$.

\item Since the $\langle -,-\rangle$ pairing on $A^{r-1}$ has signature $(1,s)$,
the standard bilinear form of the Coxeter group is positive definite
if and only if $q_{r-1}(F)  > 0$;
this corresponds to having $(r+1)^2>s(r-1)$,
i.e., if and only if
\begin{equation}\label{mdswithrs}
r=2 \text{ and } s \leq 8; \text{ or } r=3, 4 \text{ and } s \leq r+4; \text{ or } r\geq 5 \text{ and }  s \leq r+3.
\end{equation}
\end{enumerate}
\end{proposition}

Since Coxeter groups are finite exactly when the standard bilinear form is positive definite, we have:

\begin{corollary}
The Weyl group, acting on either $A^1$ or $A^{r-1}$, is finite if and only if
$(r,s)$ satisfy (\ref{mdswithrs}).
\end{corollary}

The pairs $(r,s)$ satisfying (\ref{mdswithrs}) exactly describe the cases when $Y^r_s$ is a Mori Dream Space (see \cite{Mu05} and \cite{CT}).
Hence:
\begin{corollary} \label{mds}
$Y=Y^r_s$ is a Mori Dream Space if and only if the Weyl group is finite.
\end{corollary}

It turns out that the action of $W$ on the curve classes
is isomorphic to the action on the divisor classes.
Hence all of the above statements could have been reformulated
with the divisor classes as well.
For example, the standard bilinear form of the Coxeter group
is positive definite if and only if $q_{1}(K) >0$.

\section{Basics for curves in $Y^r_s$}\label{SectionBasicsCurves}
In this section we introduce the the main terminology for curves that we will use in the paper
and we discuss general properties.
In Section \ref{diva} we explain that Weyl group orbits on divisors can be expressed algebraically using the bilinear form $\langle,  \rangle_1$ and we apply these actions to the description of the Mori cone of curves in Section \ref{appMoricone}.

\subsection{General facts}\label{gen}

Assume that $C$ is a smooth irreducible curve of genus $g$
in a smooth variety $Y$ of dimension $r$. The tangent-bundle/normal-bundle sequence for $C \subset Y$ gives
\begin{equation}\label{c1Nformula}
\deg c_1(N_{C/Y}) = 2g-2 - (C \cdot K_Y).
\end{equation}
and therefore by Riemann-Roch, we have

\begin{equation}\label{chiNformula}
\chi(N_{C/Y}) = (r-3)(1-g) - (C \cdot K_Y).
\end{equation}

In the case when $\deg c_1(N_{C/Y}) = i(r-1)$
(as is the case with an $(i)$-curve),
the two equations above reduce to

\[
(C \cdot K_Y) = 2g-2 -i(r-1) \;\;\;\text{and}\;\;\; \chi(N_{C/Y}) = (g-1-i)(1-r)
\]
which proves the following.

\begin{lemma}\label{generalcurve} For a smooth curve $C$ in $Y$ with $\deg(N_{C/Y}) = i(r-1)$,
the following are equivalent:
\begin{itemize}
\item[(a)] The genus $g$ of $C$ is zero, so that $C$ is a smooth rational curve.
\item[(b)] $\chi = \chi(N_{C/Y}) = (r-1)(i+1)$.
\item[(c)] $(C\cdot K_Y) = -2-i(r-1)$.
\end{itemize}
\end{lemma}

Without the hypothesis on the degree of the normal bundle, the two equations
(\ref{c1Nformula}) and (\ref{chiNformula}) 
allow us to easily conclude the following.

\begin{lemma} Assume $r\geq 4$.  For a smooth curve $C$,
any two of the above (a), (b), and (c) imply the third.
For $r=3$, if (a) holds, then (b) and (c) are equivalent.
\end{lemma}

The only difficulty comes in assuming (b) and (c) with $r=3$;
indeed, (b) and (c) together imply that $(r-3)g=0$ and we have no conclusion for $g$.

\begin{definition}\label{def3.5}
\begin{enumerate} Fix $i \in \{-1,0,1\}$, and let $Y=Y^r_s$.
\item An \emph{$(i)$-curve} in $Y$ is a smooth rational  irreducible curve
whose normal bundle splits as a direct sum of $\calO(i)$ line bundles.
\item A \emph{numerical $(i)$-curve} in $Y$ is an smooth rational curve $C$ such that
$(K_Y \cdot C) = -2 -i(r-1)$
\item An \emph{$(i)$-Weyl line} is a curve that is the Cremona image 
(under the Weyl group) of the proper transform of a line through $1-i$ of the $s$ points.
\item An \emph{$(i)$-curve class} is the class $c$ of an $(i)$-curve.
\item A \emph{numerical $(i)$-class} is a curve class $c \in A^{r-1}$ such that
$(K_Y \cdot c ) = -2 -i(r-1)$
\item An \emph{$(i)$-Weyl class} is a class $c \in A^{r-1}$
which is in the orbit of the class of an $(i)$-Weyl line
(equivalently, a class of an $(i)$-Weyl line). Hence
\begin{enumerate}
\item a $(-1)$-Weyl class is in the orbit of $h-e_{1}-e_{2}$; 
\item a $(0)$-Weyl class is in the orbit of $h-e_1$;
\item a $(1)$-Weyl class is in the orbit of $h$.
\end{enumerate}
\end{enumerate}
\end{definition}


Every $(i)$-curve class is a numerical $(i)$-class.
The numerical $(i)$-curves
are transformed into one another via Cremona transformations;
hence so are the numerical $(i)$-classes.

Basic questions in this situation are whether all $(i)$-curves are $(i)$-Weyl lines,
and whether all numerical $(i)$-curves are $(i)$-curves.
With the curve classes, we want to investigate
whether the numerical $(i)$-classes are all realized by classes of $(i)$-curves,
and whether the $(i)$-curve classes are $(i)$-Weyl classes.

Motivation for these definitions may be provided by the following.
Suppose $C$ is a smooth rational curve in $Y=Y^r_s$;
then we define the virtual dimension of $C$ to be
\begin{equation}\label{param}
vdim(C):=(r+1)(d+1) - (r-1)\sum_{i=1}^s m_i - 4.
\end{equation}
This comes from considering the parametrization
of the map from $\bbP^1$ to $\bbP^r$ defining the image of $C$ in $\bbP^r$,
which is defined using $r+1$ polynomials of degree $d$
(having then $(r+1)(d+1)$ coefficients),
and noting that a multiplicity $m$ point imposes $(r-1)m$ conditions.
(The last $-4$ term comes from the three-dimensional family of automorphisms of $\bbP^1$
and the homogeneity in $\bbP^r$.)
This is just a naive dimension count, assuming that all the multiplicity conditions are independent;
the actual number of parameters is at least this.

Even if this number is non-negative,
one cannot conclude that an irreducible curve exists though;
the solutions to the equations may lie at the boundary of the parameter space.

Of course this virtual number of parameters is just $\chi$ of the normal bundle:
by (\ref{chiNformula}), we have
$\chi(N_{C/Y}) = r-3-(C\cdot K_Y) = r-3 +(r+1)d - (r-1)\sum_i m_i = (r+1)(d+1) - (r-1)\sum_{i=1}^s m_i - 4$
as claimed.

If we assume that the number of parameters for such a curve $C$ is non-negative
(so that in general we can hope that
such a rational curve $C$ with these multiplicities is expected to exist)
but that it is isolated and is a reduced point in the Hilbert scheme
(it doesn't move in a family, even infinitesimally),
then we are imposing that $H^0 = H^1 = 0$ for the normal bundle,
which is equivalent to having $N_{C/Y}$ split as a direct sum of $\calO(-1)$'s,
and hence $C$ will be a $(-1)$-curve.

We then expect a finite number of smooth rational $(-1)$-curves
representing a numerical $(-1)$-class.

The following classification enables us to prove Conjecture \ref{HardConjecture} for $i = -1$ in some special cases of interest.

\begin{proposition}\label{s leq r+4 (-1)-curves}
\begin{itemize}
\item[(a)] If $r \geq 3$ and $s \leq r+4$ then every $(-1)$-Weyl line is either a line through two points or the rational normal curve of degree $r$ through $r+3$ points.
\item[(b)] If $r = 2$, or $r \geq 3$ and $s \leq r+4$, then every $(-1)$-Weyl line is a $(-1)$-curve.
\end{itemize}
\end{proposition}

\medskip\noindent{\bf Proof:}
To prove (a), start with the line through two points (which we may assume to be $p_1$ and $p_2$).
Apply the first standard Cremona transformation, by choosing $r+1$ base points.
If $p_1$ and $p_2$ are both among those base points, the line is contracted, and we do not produce a $(-1)$-Weyl line.
If only one of the two points are among those base points, the line is transformed to a line through two points.
If neither of the two points are among those base points, the line is transformed to the rational normal curve through all of the $r+3$ points.

Now analyze the rational normal curve through $r+3$ points.  
Again choose $r+1$ base points and apply the standard Cremona transformation.
If those $r+1$ points are a subset of the $r+3$ points on the curve,
then the curve is transformed to the line through the remaining two points.
If not, then since $s \leq r+4$, we must have $s=r+4$
and the $r+1$ base points consist of $r$ points on the curve,
and the extra point not among the $r+3$ points on the curve.
In that case the curve is transformed to a rational normal curve again.

This proves (a), and statement (b) is clear for $r=2$.
For $r\geq 3$, it follows by first noting that
the line starts with normal bundle equal to $\mathcal{O}(1)^{\oplus(r-1)}$ in $\mathbb{P}^r$
and the rational normal curves starts with normal bundle equal to $\mathcal{O}(r+2)^{\oplus(r-1)}$ in $\mathbb{P}^r$.
Since blowing up a point twists the normal bundle by $\mathcal{O}(-1)$, 
both become $(-1)$-curves in $Y^r_s$.

\rightline{QED}

The classes of irreducible curves have additional constraints on them,
some of which are captured in the following.

\begin{lemma}\label{irreducible}
Suppose $c=dh-\sum_i m_i e_{i}$ is the class of an irreducible curve $C$
not inside one of the exceptional divisors $E_i$.  Then
\begin{itemize}
\item[(a)] for each $i$, $d \geq m_i \geq 0$;
\item[(b)] If $d \geq 2$ then for each $i\neq j$, $d \geq m_i + m_j$;
\item[(c)] if the multiplicities are in descending order $m_1 \geq m_2 \geq \ldots \geq m_s$,
and if for some $k \leq r$ we have $d < \sum_{i=1}^k m_i$,
then $m_{k+1}=\cdots=m_s = 0$.
\end{itemize}
\end{lemma}

\medskip\noindent{\bf Proof:}
If such a curve $C$ exists with this class, then the intersection with $E_i$ must be non-negative, so each $m_i \geq 0$.
In addition, intersecting $C$ with divisors in the class $H-E_i$
shows that $d \geq m_i$ since $|H-E_i|$ has no base locus.
This proves (a).
For (b), if $d < m_i+m_j$,
then the intersection of $C$ with every hyperplane through $p_i$ and $p_j$ is negative,
so that if $C$ is irreducible it must be inside every such hyperplane,
and therefore inside their intersection, which is the line joining the two points;
hence $d=1$ (and $c$ must be the class of that line).
Finally (c) is similar; if the inequality holds then $C$ lies inside the linear space
spanned by those first $k$ points, and so cannot have positive multiplicity at
any other point outside that linear space (where all other points are by the generality).

\rightline{QED}

\subsection{Projections}\label{projections}
Projections offer an interesting perspective and tool to study these curves.
If $C$ is an irreducible curve in $\bbP^r$ 
whose proper transform in $Y^r_s$ has class 
$c = dh-\sum_{i=1}^s m_ie_i \in A^{r-1}(Y^r_s)$, 
then we may consider the projection $\pi$ from any one of the multiple points
(say $p_1$ with multiplicity $m_1$)
and obtain the curve $\pi(C) \subset \bbP^{r-1}$
whose proper transform in $Y^{r-1}_{s-1}$ has class 
$\pi(c) = (d-m_1)h-\sum_{i=2}^s m_i e_i \in A^{r-2}(Y^{r-1}_{s-1})$.

We've seen above that $\chi(N_{C/Y})$ is determined by the class of $C$,
and it can happen that $\chi(N)$ is non-negative for $C$ but negative for $\pi(C)$.
This is an indication that,
although the $\chi(N)$ computation suggests that a rational curve exists with that class,
the projection is not expected to exist
(and hence the original curve actually won't either).

For example, suppose $c$ is a numerical $(-1)$-class
so that $(r+1)d-(r-1)\sum_{i=1}^s m_i =3-r$;
this is the condition that $\chi(N_{C/Y}) = 0$ if the irreducible curve $C$ exists.
In this case $\chi(N_{\pi(C)}) = r(d-m_1)-(r-2)\sum_{i=2}^{s} - 3 + (r-1)$,
and one computes that
\[
\chi(N_C) - \chi(N_{\pi(C)}) = d +m_1+1 - \sum_{i=2}^{s-1} m_s.
\]
Therefore if this quantity is positive we don't expect $\pi(C)$ to exist,
and hence shouldn't expect $C$ to exist.
We would generally apply this when $m_1$ is the largest multiplicity,
obtaining the criterion that
\begin{equation}\label{projectioninequality}
d+m_1 \leq \sum_{i=2}^{s} m_s
\end{equation}
for $C$ to be expected to exist.
This essentially says that, for a given degree, no single multiplicity should be too large.

It is easy to see that projections commute with the Cremona transformation:
\begin{proposition}\label{Cremonaprojectioncommuteprop}
Let $c=dh-\sum_im_ie_i$ be a class in $A^{r-1}(Y^r_s)$.
We have the Cremona transformation $\phi:A^{r-1}(Y^r_s) \to A^{r-1}(Y^r_s)$ 
on $Y^r_s$
(based at the first $r+1$ points)
and also the Cremona transformation 
$\phi:A^{r-2}(Y^{r-1}_{s-1}) \to A^{r-2}(Y^{r-1}_{s-1})$
(based at the first $r$ points).
If we denote by $\pi$ the projection from the first point, we have
\[
\pi(\phi(c)) = \phi(\pi(c)).
\]
\end{proposition}

\begin{pf}
Using the formulas from Proposition \ref{CremonaAY}(b),
we see that if we define $t=d-\sum_{i=1}^{r+1} m_i$,
then
\[
\phi(c) = (d+t) h - \sum_{i=1}^{r+1}(m_i+t) e_i - \sum_{i > r+1} m_i e_i,
\]
and so
\[
\pi(\phi(c)) = (d+t-(m_1+t)) h - \sum_{i=2}^{r+1} (m_i+t) e_i - \sum_{i > r+1} m_i e_i.
\]
Now $\pi(c) = (d-m_1)h - \sum_{i=2}^s m_i e_i$, so that if we define $t' = (d-m_1)-\sum_{i=2}^{r+1} m_i$,
we have
\[
\phi(\pi(c)) = (d-m_1+t')h - \sum_{i=2}^{r+1} (m_i+t') e_i - \sum_{i>r+1} m_ie_i.
\]
The result follows by noting that $t'=t$.
\end{pf}

We note that if $I$ is any subset of $r+1$ indices for the $s$ points of $Y^r_s$,
and $i \in I$,
then we may denote the Cremona transformation based at the points with indices in $I$ by $\phi_I$.
If we denote the projection from the $i$-th point by $\pi_i$,
then the same proof as above shows that
\[
\pi_i(\phi_I(c)) = \phi_{I-\{i\}}(\pi_i(c)).
\]

We present the following Lemma, which will be useful later.

\begin{lemma}\label{projection-1lemma}
Assume that $r \geq 3$ and $s \leq r+4$.
Let $C$ be a $(-1)$-Weyl line in $Y^r_s$.
Since $C$ is a $(-1)$-Weyl line, there is a Cremona transformation $A$ that takes $C$ to a line through two points.  
Factor $A$ as $a_1a_2\ldots a_k$ minimally, so that each $a_i$ reduces the degree.
Then there is a base point of $a_k$ such that projecting from that point
gives a $(-1)$-Weyl line class in $Y^{r-1}_{s-1}$.
\end{lemma}

\medskip\noindent{\bf Proof:}
Let us argue by induction on $k$.
If $k=1$, then $C$ is obtained from the line through two points by applying one standard Cremona transformation based at $r+1$ of the $s$ points.
If the two points of the line are among the $r+1$ points we have a contradiction: the line is contracted.
If exactly one of the two points of the line are among the $r+1$ points, 
then $C$ is again a line through two points, and this means that $k=0$ not $k=1$.
Hence we may suppose that neither of the two points of the line 
are among the $r+1$ points.
In this case $C$ is the rational normal curve of degree $r$ in $\bbP^r$ through $r+3$ points.
Hence the projection of $C$ from one of those $r+3$ points having multiplicity one is a rational normal curve of degree $r-1$ in $\bbP^{r-1}$ as is well-known.

Suppose now the statement is true for $k$, and let us show it for $k+1$.

Let $\pi$ be the projection from one of the base points of the first standard Cremona $a_{k+1}$ in the factorization of $A$.
If $c$ is the class of $C$, then Proposition \ref{Cremonaprojectioncommuteprop} shows that
$B(\pi(c)) = \pi(a_{k+1}(c)) $
where $B$ is the standard Cremona transformation of $\bbP^{r-1}$ 
based at the remaining $r$ points,
the images of the base points of $a_{k+1}$ which are not the projection point.
Since $B$ is an involution, we have therefore
$\pi(c) = B(\pi(a_{k+1}(c)))$.
Now by induction on $k$,
the class $a_{k+1}(c)$ represent a curve $D$ in $\bbP^r$ which satisfies the Lemma,
i.e., the projection from one of the $r+1$ base points of $a_k$ is a $(-1)$-Weyl line.
Since $s \leq r+4 \leq 2r+1$, 
there is at least one point of intersection between the $r+1$ base points of $a_{k+1}$ 
and of $a_k$. 
This is the point we choose to project from with the projection $\pi$.

In that case by induction $\pi(a_{k+1}(c))$ is the class of a $(-1)$-Weyl line.
Therefore so is $B(\pi(a_{k+1}(c)))$.
Therefore so is $Q(c)$.

\rightline{QED}

\subsection{The Planar Case.}\label{section planar}

We now analyse $(i)$-curves in the planar case $Y=Y^2_{s}$.

We note that in this case $r=2$

We note that in this case, $A^1 = A^{r-1}$ and all three bilinear forms
$(-\cdot -)$, $\langle -,-\rangle_1$, 
and $\langle -,-\rangle = \langle -,-\rangle_{r-1}$ are equal,
and the curve class $F = -K$.

For a plane curve class $c$, the common invariants can be determined from
the intersection form:
the arithmetic genus $p_a(c) = (2+(c\cdot c)+ (c\cdot K_Y))/{2}$
and the Euler characteristic $\chi(c) = \chi(\calO_Y(c)) = ((c\cdot c)- (c\cdot K_Y))/{2}$.
With these formulas, we have immediately the following:

\begin{proposition}\label{planar}
Let $c\in \text{Pic}(Y)$ be an arbitrary curve class and fix $i\in \mathbb{Z}$.
Then any two of the following statements imply the others:
\begin{enumerate} 
\item $p_a(c)=0$
\item $\chi(c)=2+i$
\item $(c \cdot c)=i$
\item $(c \cdot K_Y) = -2-i$.
\end{enumerate}
\end{proposition}

A class satisfying the fourth condition is, by definition, a numerical $(i)$-class.
If $C$ is an irreducible curve with class $c$,
then the first condition says that $C$ is smooth and rational,
and the third condition on the self-intersection would say that $C$ is an $(i)$-curve.

However these numerical conditions do not imply that a curve $C$
with a class satisfying the above conditions must be irreducible:
\begin{example}\label{exp}
Consider the planar divisor of degree $5$ with two triple points and eight simple points.
It satisfies the conditions of Proposition \eqref{planar} with $i= -1$
but it is not represented by an irreducible curve:
the proper transform of the line through the two triple points splits off from this system.
\end{example}

\begin{remark}
If $i\geq -1$, and if $C$ is any class satisfying any two of the above equivalent conditions,
with positive degree, then we have 
\[
\chi(C)=h^0(C)-h^1(C)=2+i\geq 1
\]
(because the $H^2$ term must vanish if $C$ has positive degree).
We conclude that any such class is effective.
\end{remark}

The next result is the planar case of Theorem \ref{div} whose proof is a simple consequence of the Max Noether inequality.

\begin{theorem}\label{Y2sirreducible}
Let $C$ be an irreducible curve in $Y$ and fix $i\in \{-1,0,1\}$.
Then $C$ is an $(i)$-curve if and only if $C$ is an $(i)$-Weyl line.
\end{theorem}

The next result shows that for planar Mori Dream Spaces
the irreducibility assumption may be relaxed.

\begin{theorem}\label{thm1}
Let $s \leq 8$, so that $Y^2_s$ is a Del Pezzo Surface, and fix $i\in \{-1,0,1\}$.
Then an effective divisor $C$ is linearly equivalent to an $(i)$-curve if and only if the conditions of Proposition \ref{planar} hold, i.e.
\begin{equation}\label{equa2}
(C \cdot C) =i \;\;\text{ and } \;\; (C \cdot K_Y) = -2-i.
\end{equation}
If $i=-1$ then $C$ is a $(-1)$-curve itself.
\end{theorem}

\begin{proof}
In the Del Pezzo case, $-K$ is ample,
and so the $(C \cdot K)$ condition says that 
$C$ has anticanonical degree equal to $1$, $2$, or $3$.
If $(C\cdot -K) = 1$ (the $i=-1$ case), then $C$ must be irreducible and smooth,
and then the $(C\cdot C)$ condition implies that $C$ is rational,
hence a $(-1)$-curve.

In case $C$ has anticanonical degree greater than $1$,
we conclude in a similar way if $C$ is irreducible.
Hence assume $C$ is not irreducible, with anticanonical degree $2$ or $3$.
In that case $C$ must split as $C = kG + J$ where $G$ has anticanonical degree $1$,
(and hence as above is a $(-1)$-curve),
$1 \leq k \leq 2$, and $G$ and $J$ are distinct.
If $k=2$, then $J$ is a $(-1)$-curve as well, and in that case
$(C \cdot C) = 4(G^2) + (J^2) + 4(G\cdot J) = 4(G\cdot J)-5$
which cannot be equal to $0$ or $1$ as required.

We can therefore assume $k=1$.
In the case $i=0$ with anticanonical degree $2$,
then $C = G+J$ with $G$ and $J$ distinct $(-1)$-curves,
and the self-intersection condition implies $(G\cdot J)=1$.
In that case $C$ moves in a pencil whose general member is a $(0)$-curve.

In the case $i=1$ with anticanonical degree $3$,
then $J$ has anticanonical degree $2$;
by the above analysis we may assume that $J$ moves in a pencil
whose general member is a $(0)$-curve.
In that case
we have $1 = (C \cdot C) = (G+J \cdot G+J) = 2(G\cdot J)-1$,
which gives $(G\cdot J)=1$ and then $C$ moves in a linear system
whose general member is an $(1)$-curve.
\end{proof}


\begin{proposition}\label{del}
Suppose that $i\in \{-1,0\}$,and $C$ is an irreducible curve in $Y^2_s$
whose class is a numerical $(i)$-class.  Then the possibilities are:
\begin{enumerate}
\item ($i=-1$ case): $C$ is either a $(-1)$-Weyl curve or $C$ is a cubic with class $F$ in $Y^2_8$.
\item ($i=0$ case): $C$ is either a $(0)$-Weyl curve, or $C$ is a cubic with class $F$ in $Y^2_7$, or $C$ is a sextic with class $2F$ in $Y^2_8$.
\end{enumerate}
\end{proposition}

\begin{proof}
Suppose that $C$ is an irreducible curve on $Y$ with positive degree
whose class is a numerical $(i)$-class, i.e., $(C\cdot K_Y) = -2-i$.
Write the class of $C$ as $c = (d;m_1,m_2,\ldots,m_s)$
with the multiplicities in decreasing order.
If $d=1$ then $-3+\sum_i m_i = -2-i$ so $\sum m_i = 1-i \leq 2$;
hence we must have an $(i)$-Weyl line.
We may therefore assume that $d \geq 2$;
if $d < m_1+m_2+m_3$,
we may perform a Cremona transformation and reduce the degree.
Hence we may assume that $d \geq m_1+m_2+m_3$.
In this case the numerical condition on $(C\cdot K)$ gives 
$
0\leq (m_1-m_4)+(m_2-m_5)+(m_2-m_6)+(m_3-m_7)+(m_3-m_8)+m_1 \leq 2+i 
$
If $m_1=0$ then all $m_i=0$ and this is impossible.
If $m_1=1$ then assume that $m_k=1$ but $m_{k+1} = 0$;
then the intersection with $K$ gives that $3d-k=2+i$,
so that the only solutions that respect the inequalities are
$(1;1)$ and $(3; 1^7)$ for $i=0$ and $(3; 1^8)$ for $i=-1$.
The case $m_1 \geq 2$ gives the unique solution $(6; 2^8)$ for $i=0$.
\end{proof}

Theorem \ref{nms2} of Section \ref{infinite} in the planar case  reproves via the theory of movable curves  that if $Y$ is not a del Pezzo surface then it is not a Mori Dream Space,
because there are infinitely many $(0)$- and $(1)$-Weyl lines.

\subsection{The Divisor Case}\label{diva}

Certain algebraic equations for the Weyl orbits of hyperplanes
using the Dolgachev-Mukai intersection pairing
were established in \cite{DP}.
In this section Corollary \ref{div} extends results from $(-1)$-Weyl hyperplanes
to $(1)$-Weyl hyperplanes. 
As we will observe in this paper, $(-1)$-Weyl lines in higher dimension
cannot be characterized by the quadratic and linear bilinear form.

Since there is no arithmetic formula
to express a rationality condition for divisors or curves
in projective space of higher dimension,
in \cite{DP} the definition of divisorial classes
was formulated algebraically similarly to conditions
(3) and (4) of Proposition \ref{planar}.
We introduce below $(0)-$ and $(1)-$ Weyl hyperplanes.

\begin{definition}\label{def} Let $i\in \{-1,0,1\}$.
\begin{enumerate}
\item An \emph{$(i)$-Weyl hyperplane} is a divisor in the Weyl orbit
of a hyperplane through $r-1-i$ points.

\item A \emph{numerical $(i)$-divisorial class} is a divisor class $[D] \in A^1$
such that  $\langle  [D], [D] \rangle_1 =i$
and $\frac{1}{r-1}\langle  [D], -K_Y \rangle_1=2+i$.
\end{enumerate}
\end{definition}

It is obvious that the class of an $(i)$-Weyl hyperplane
is a numerical $(i)$-divisorial class.
They are equivalent for irreducible effective divisors.
This is based on the Noether inequality for divisors;
for the proof of the following see \cite[Theorem 0.4]{DP}:

\begin{proposition} 
Assume $i\in \{-1,0,1\}$.
Let $D$ be an irreducible effective divisor whose class is a numerical $(i)$-divisorial class,
and assume that is not a hyperplane through $n-1$ or $n-2$ base points.
Then the divisor $D$ is not Cremona reduced, i.e.
if we write $[D] = dH-\sum_i m_i E_i$ and order the multiplicities in decreasing order,
then $(r-1)d < \sum_{i=1}^{r+1} m_i$.
\end{proposition}

\begin{corollary}\label{div}
Assume $i\in \{-1,0,1\}$.
Let $D$ be an irreducible effective divisor
whose class is a numerical $(i)$-divisorial class.
Then $D$ is an $(i)$-Weyl hyperplane.
\end{corollary}

\subsection{Applications to Mori cone of curves}\label{appMoricone}

Let us introduce
$\mathcal{Z}_{\geq 0}$
as the cone of curve classes in $A^{r-1}(Y)$ 
that meet all numerical $(0)$-divisorial classes non-negatively.
Define $\mathcal{Z}_{\geq 0}\langle -1 \rangle$
to be the cone generated by $\mathcal{Z}_{\geq 0}$ and all $(-1)$-curves in $Y^r_s$.

Then we have the following inclusion of the Mori cone of curves on $Y$:

\begin{theorem}\label{cor}
The cone of classes of effective curves in $Y^r_s$ is a subcone of 
$\mathcal{Z}_{\geq 0}\langle -1 \rangle$.

\end{theorem}

\section{$(i)$-curves in Mori Dream Spaces $Y^r_s$, $r\geq 3$}\label{-1curvesMDS}
We now turn our attention to $(i)$-curves in $Y^r_s$ with $r\geq 3$,
with attention on the Mori Dream Space cases of (\ref{mds});
for $r \geq 3$ these cases are $s \leq r+4$ for $r = 3,4$ and $s \leq r+3$ for $r\geq 5$.
We note that in all these cases, Theorem \ref{ConjectureProofTheorem} holds,
and every $(i)$-Weyl line is an $(i)$-curve for each $i \in \{-1,0,1\}$.

We will assume that $s \geq r+2$;
for lower values of $s$ there is at most one Cremona transformation,
which commutes with the symmetric group,
and the situation is easy to analyze.

For $s=r+2$ it is elementary to see that
the only $(-1)$-Weyl lines are the lines through two of the points.
The only $(0)$-Weyl lines are the lines through one of the points,
and the conics through all $r+2$ points;
the latter only happens for $r=2$ though, since a conic must live in a plane.
The only $(1)$-Weyl lines are the general lines,
and again conics through $r+1$ points (only in the case $r=2$ again).

This case $s=r+2$ is satisfactorily treated in Proposition \ref{cremona},
and shows that there are no numerical $(i)$-classes other than the ones noted above.

We will concentrate therefore on the cases $s \geq r+3$. We  will introduce here the terminology that we will use throughout the section.

We will denote by $c=(d;m_1,\ldots,m_s)_r$ a class in $A^{r-1}$
and we will employ exponential notation to indicate repeated multiplicities.

\begin{definition} We say that the class $c \in  A^{r-1}(Y)$ with $d \geq 2$ 
is \emph{Cremona reduced} if 
the multiplicities are arranged in decreasing order and
$d \geq m_1 + m_2 + \dots + m_{r+1}$.
\end{definition}

The inequality above is the condition that
no Cremona transformation centered at $r+1$ of the points
will reduce the degree of the class.
(In the case of the line through two points, the Cremona transformation contracts it.)

We recall that the canonical divisor is fixed under the Weyl group action,
and that irreducibility and effectivity of curves of degree at least 2
is preserved under the Weyl group action.
Throughout this section we will use the following observation:
\begin{remark}\label{reduced}
If  $c\in A^{r-1}$ is the class of
an effective and irreducible curve $C$ of degree at least $2$, 
then we can reduce $C$ to an Cremona reduced effective, irreducible curve
with class $f\in A^{r-1}$  and $(K_Y \cdot c) =(K_Y \cdot f)$.
\end{remark}

This focuses our attention on the Cremona reduced classes.
The following will be useful in our analyses.

\begin{lemma}\label{CRnumjclass}
Suppose $j \in \{-1,0,1\}$ and that $c$ is a Cremona reduced numerical $(j)$-class
with positive degree and non-negative multiplicities.
Then
\[
(r-1)\sum_{i=r+2}^s m_i \geq -2 - j(r-1)+2\sum_{i=1}^{r+1}m_i.
\]

\end{lemma}

\medskip\noindent{\bf Proof:}
Write $c = (d; \underline{m})$ with multiplicities in decreasing order.
The numerical condition is that
\[
(r+1)d - (r-1)\sum_{i=1}^s m_i = 2+j(r-1)
\]
and therefore
\[
(r-1)\sum_{i=1}^s m_i = -2-j(r-1)+(r+1)d \geq -2 - j(r-1)+ (r+1)\sum_{i=1}^{r+1}m_i
\]
using the Cremona reduced assumption.  Subtracting $(r-1)\sum_{i=1}^{r+1}m_i$
from both sides gives
\[
(r-1)\sum_{i=r+2}^s m_i \geq -2 - j(r-1)+2\sum_{i=1}^{r+1}m_i.
\]

\rightline{QED}

\subsection{The case $s = r+3$}\label{iclasses}
We first address the case of $s = r+3$ with $r\geq 2$.
It is an easy computation to show that the only $(-1)$ Weyl lines in Mori Dream Spaces and $Y^5_9$
are the lines through two points and the rational normal curves of degree $r$ through all $r+3$ points:
 (and all permutations). To see this, one computes from the bottom up:
applying Cremona transformations to these classes gives these classes back,
no matter which $r+1$ points one chooses.  In particular, on these spaces the $(-1)$-curves are the $\binom{s}{2}$ lines through two points and rational normal curves of degree $r$ passing through $r+3$ points. Similarly, the $(0)$-Weyl lines are just the lines through one of the points,
and the rational normal curves through all but one of the points.

\begin{proposition}\label{cremona}
Let $j\in \{-1, 0\}$, $r \geq 3$, and $s\leq r+3$. Let $C$ be an irreducible curve in $Y^r_s$.
Then $[C]\in A^{r-1}$ is a numerical $(j)$-class if and only if $C$ is an $(j)$-Weyl line.
\end{proposition}

\medskip\noindent{\bf Proof:}
Since the class of every $(j)$-Weyl line is a numerical $(j)$-class,
it suffices to prove the other direction.

Let $c=[C]$ have the form $d h - \sum_i m_i e_i$;
if $d=1$ we are done.  We assume that $d \geq 2$
and with multiplicities ordered decreasingly.

By Remark \ref{reduced} we can assume that $c$ is Cremona reduced,
and that it satisfies the inequalities $d \geq m_i \geq 0$
since $C$ is irreducible.

Now Lemma \ref{CRnumjclass} gives that

\[
(r-1)(m_{r+2}+m_{r+3}) \geq 2 \sum_{i=1}^{r+1} m_i +j(1-r)-2 = 2 \sum_{i=3}^{r+1} m_i + 2m_{1}+2m_{2}+j(1-r)-2 
\]
and since $\sum_{i=3}^{r+1} m_i \geq (r-1)m_{r+2} \geq (r-1)m_{r+3}$,
the right-hand side of the above is at least
$(r-1)(m_{r+2}+m_{r+3}) + 2m_{1}+2m_{2}+i(1-r)-2$.
Subtracting  $(r-1)(m_{r+2}+m_{r+3})$ from both sides now gives
\[
0 \geq 2m_{1}+2m_{2}+j(1-r)-2.
\]

If $j=-1$, this says that $3-r \geq 2m_1+2m_2$,
which implies all multiplicities are non-positive if $r\geq 3$,
a contradiction if $d\geq 2$.
(If $r=2$ we must have $2m_1+2m_2 \leq 1$
so again we must have all multipliicties $m_i=0$,
which again gives a contradiction:
we would then have to have $3d = 1$ and this is impossible.)

If $j=0$ then we have $2(m_1+m_2) \leq 2$,
and therefore $m_1=1$ and $m_2=0$.
This forces $d=1$, and we have the $(0)$-Weyl line.

\rightline{QED}

It is tempting to try to use this proposition to conclude that,
in these ranges of parameters,
all numerical $(-1)$-classes are $(-1)$-Weyl classes.
However this is false:
we may indeed be able to reduce the degree,
but the result may be a class whose degree becomes negative
(or a multiplicity becomes negative).

The example below shows that the \textit{irreducibility} assumption is important:
it ensures that Cremona transforms stay with non-negative parameters.

\begin{example}
Consider the example of the class $(13;4,3^6)_4$
 of degree $13$ in $\bbP^4$ with one point of multiplicity $4$
and six points of multiplicity $3$.
This is an effective class in $\bbP^4$ with $7$ points
and $(-K\cdot C) = 5 \cdot 13 - 3\cdot (4 + 3 \cdot 6) = -1$ so $C$ is a numerical $(-1)$-curve.
However it is not a Weyl class because it is not irreducible.
Applying a Cremona transformation to five of the points including the multiplicity $4$ point
gives the class $(4;3^2,1,0^4)_4$
which cannot be the class of an irreducible curve,
by Lemma \ref{irreducible}.
Indeed, any curve $C$ having this class must contain the line through the two triple points,
and the residual class is $(3;2^2,1,0^4)_4$,
which also must contain that line;
the second residual is $(2;1^3,0^4)_4$
which is the class of a net of conics in the plane spanned by the three points of multiplicity one.
\end{example}

We note that for the example above, $c:=(13;4,3^6)_4$ then $\langle c, c \rangle = -41$
which is not equal to $3-2r = -5$.
Hence it is not a $(-1)$-Weyl line class using this criterion, either.

We can replace the assumption that the class comes from an irreducible curve
by using the quadratic invariant $\langle c,c\rangle$ in this case,
and obtain a purely numerical condition for a class to be a $(-1)$-Weyl line class:

\begin{proposition}\label{rplus3points}
Suppose $r\geq 3$ and $s = r+3$.
Let $c \in A^{r-1}$ be a class with positive degree and non-negative multiplicities.
Then $c$ is a $(-1)$-Weyl line class if and only if
\[
\langle c, F \rangle = 3-r \text{ and } \langle c,c\rangle = 3-2r.
\]
\end{proposition}

\medskip\noindent{\bf Proof:}
Since the bilinear form and the class $F$ are invariant under the Cremona transformations,
one implication holds, since those are the values for a line through two points.
Therefore it is enough to prove that a class with the given quadratic and linear invariants
is a $(-1)$ Weyl line.
Let $c = (d;\underline{m})_r$ be such a class, and write
$N_1 = \sum_i m_i$ and $N_2 = \sum_i m_i^2$.
Our assumptions are then that $(r+1)d-(r-1)N_1 = 3-r$
and $d^2-(r-1)N_2 = 3-2r$.

We first claim that $d \leq r$.
For a fixed degree $d$, the quantity $N_1$ is then determined,
namely $N_1 = ((r+1)d + (r-3))/(r-1)$.
If all of the multiplicities are equal,
then each one of them would be equal to $m=N_1/(r+3)$.
The value of $N_2$ would be minimized if all multiplicities were equal, i.e.,
$N_2 \geq \sum_{i=1}^{r+3} m^2 = (r+3)(N_1/(r+3))^2 = N_1^2/(r+3)$.
Therefore
\begin{align*}
3-2r &= d^2-(r-1)N_2 \leq d^2 - (r-1) N_1^2/(r+3) \\
&= d^2 - \frac{r-1}{r+3}(((r+1)d + (r-3))/(r-1))^2 \\
&= d^2 - \frac{1}{(r+3)(r-1)} ((r+1)^2 d^2 + 2(r+1)(r-3) d + (r-3)^2) \\
&= \frac{1}{(r+3)(r-1)} ( ((r+3)(r-1)-(r+1)^2)d^2 - 2(r+1)(r-3) d - (r-3)^2 ) \\
&= \frac{-1}{(r+3)(r-1)} ( 4d^2 + 2(r+1)(r-3) d + (r-3)^2 ).
\end{align*}
Now the upper bound quantity here when $d=r$ is
$\frac{-1}{(r+3)(r-1)} (4r^2 + 2r(r+1)(r-3) + (r-3)^2 )
= \frac{-1}{(r+3)(r-1)}((r-1)(r+3)(2r-3) = 3-2r$
and the quadratic function in $d$ only increases for $d \geq r$,
so that if $d > r$, then $ \langle c, c \rangle < 3-2r$ which is a contradiction.
Hence we conclude that $d \leq r$.

Since all the multiplicities are non-negative integers,
we have $m_i \leq m_i^2$ for every $i$,
so that $N_1 \leq N_2$ (with equality only if all multiplicities are $0$ or $1$).
Since $(r-1)N_1 = (r+1)d+r-3$ and $(r-1)N_2 = d^2+2r-3$, we conclude that
$(r+1)d+r-3 \leq d^2+2r-3$,
which simplifies to $(d-1)(d-r) \geq 0$.
Hence $d$ cannot lie in the open interval $(1,r)$,
and so the only possibilities are narrowed down to $d=r$ and $d=1$.

If $d=r$ then the values for $ \langle c, F \rangle$ and  $\langle c, c \rangle$ imply that $N_1 = N_2 = r+3$.
Hence all multiplicities are either $0$ or $1$; since the sum is the number of points $r+3$, they must all be one, and we have the class of the rational normal curve.

Similarly if $d=1$ we must have $N_1 = N_2 = 2$, and this leads to the line through two points.

\rightline{QED}

We can summarize the situation for irreducible $(-1)$- and $(0)$-curves 
when $s = r+3$ in the following.

\begin{theorem}\label{class2} Let $i\in \{-1,0\}$, $r \geq 3$, 
and  $c\in A^{r-1}(Y^{r}_{r+3})$ be the class of an irreducible curve.  
The following are equivalent:
\begin{enumerate}
\item $c$ is the class of a $(i)$-curve.
\item $c$ is a $(i)$-Weyl line class.
\item $\langle c, c \rangle=1+(i-1)(r-1)$ and $\langle c, F \rangle =2+i(r-1)$.
\item $c$ is a numerical $(i)$-class.
\end{enumerate}
\end{theorem}

\begin{proof}
The invariance of the bilinear form gives us that (2) implies (3),
and we always have that either (1) or (3) implies (4).
Theorem \ref{ConjectureProofTheorem} proves that (2) implies (1).

Since $s=r+3$, Proposition \ref{cremona} proves that (4) implies (2),
which completes the equivalencies.
\end{proof}

We study now $(1)$-curves in Mori Dream spaces with $s = r+3$
 and we recall the anticanonical curve class $F:=(r+1)h-\sum_{i=1}^{s}e_i$.

\begin{proposition}\label{icremona}
Let $r \geq 3$, $s=r+3$, and let $C$ be an irreducible curve in $Y^r_s$.
Then $c=[C]\in A^{r-1}$ is a numerical $(1)$-class if and only if
$C$ is a $(1)$-Weyl line
or $2|r+1$ and $c$ is in the Weyl orbit of a class of the form
$c = mF + c'$
where $m \leq (r+1)/4$ and $c' = (e;\underline{n})$
where $e=(r+1)/2 - 2m$, $n_i = 0$ for $i \geq r-1$, and $e = \sum_i n_i$.

\end{proposition}

\begin{proof}
We first remark that one direction is clear.
Certainly if $C$ is a $(1)$-Weyl line,
we have the numerical condition.
Suppose $c$ has the form $mF+c'$ as in the statement.
If we write $c = (d;\underline{m})$,
then $d = (r+1)m+\sum_i n_i = (r+1)m+e$
and $m_i = m+n_i$ for each i.
Hence
\begin{align*}
(-K \cdot c) &= (r+1)d-(r-1)\sum_i m_i \\
&= (r+1)^2m + (r+1)\sum_i n_i - (r-1)((r+3)m+\sum_i n_i) \\
&= 4m + 2\sum_i n_i  = 4m+2e = r+1
\end{align*}
as required.

To prove the other direction, we assume we have the numerical condition,
and that the class $c$ is Cremona reduced with decreasing multiplicities.
If $d=1$ then the only possibility is the class $h$,
so we may assume $d\geq 2$,
and then we must have $m_1 \geq 1$ as well.

We will prove in this case that $c = mF+c'$ as in the statement of the Proposition.
The condition for a numerical $(1)$-class is
\[
\langle c, F \rangle = (-K \cdot c) =  (r+1)d-(r-1)\sum_{i=1}^{r+3} m_i = r+1
\]
which can be written as
\begin{equation}\label{num(1)class}
(r-1)(d-1-\sum_{i=1}^{r+1} m_i) + (d-1-(r-1)m_{r+2})+(d-1-(r-1)m_{r+3})=0.
\end{equation}

Since $c$ is Cremona reduced we have $d \geq \sum_{i=1}^{r+1} m_i$
which is at least $m_1+m_2+(r-1)m_{r+2}$ since the multipliciities decrease.
Therefore in (\ref{num(1)class}),
the last two parenthesis are non-negative
and therefore $d-1-\sum_{i=1}^{r+1}m_i$ is non-positive.
However this forces $t =d-\sum_{i=1}^{r+1}m_i \in \{0,1\}$
since $c$ is Cremona reduced.

We distinguish two cases\\
\medskip\noindent \textbf{Case $t=1$}

Now (\ref{num(1)class}) gives that all three terms are zero,
and hence we must have $d=1+(r-1)m_{r+2}=1+(r-1)m_{r+3}$.

We have
\[
0=t-1 = d-1 - \sum_{i=1}^{r-1} m_i -(m_r+m_{r+1})
\leq d-1 -(r-1)m_{r+2} -(m_r+m_{r+1}) = -(m_r+m_{r+1})
\]
which forces $m_r=0$, hence $m_{r+2}=0$ also, so that $d=1$,
a contradiction.

\medskip\noindent \textbf{Case $t=0$} 

In this case  (\ref{num(1)class}) implies
\[
(\sum_{i=1}^{r+1} m_i -(r-1) m_{r+2}) + (\sum_{i=1}^{r+1} m_i -(r-1)m_{r+3}) = r+1
\]
which can be reorganized as
\begin{equation}\label{equua}
\sum_{i=1}^{r-1}(2m_i-m_{r+2}-m_{r+3})+2m_r+2m_{r+1} =r+1.
\end{equation}

Since the $m_i$'s are decreasing,
all the terms in the sum are non-negative, and decreasing.
Hence the last one, $2m_{r-1} - m_{r+2}-m_{r+3}$ is the smallest.
We distinguish three cases:

\medskip\noindent{\bf Case A:}  $2m_{r-1} - m_{r+2}-m_{r+3} \geq 2$.
Now
\[
r+1 = 2m_r+2m_{r+1} + \sum_{i=1}^{r-1}(2m_i-m_{r+2}-m_{r+3})
\geq 2(r-1)
\]
which forces $r=3$, $m_r=m_{r+1}=0$,
and $2m_i-m_{r+2}-m_{r+3}=2$ for all $i \leq r-1$.
Since $m_r=0$, so are $m_{r+2}$ and $m_{r+3}$,
and hence we have $m_i=1$ for all $i \leq r-1$, all other $m_j=0$.
Now $t=0$ gives $d=2$ and we have $c = (2;1^2,0^4)_3$
which is not a numerical $(1)$-class.

\medskip\noindent {\bf Case B:} $2m_{r-1} -m_{r+2}-m_{r+3} = 1$.
In this case the sum in (\ref{equua}) is at least $r-1$, so that
we must have $2m_r+2m_{r+1} \leq 2$.
This forces $m_{r+1}=0$, and hence $m_{r+2}=m_{r+3}=0$ as well.
But then our Case B assumption gives $2m_{r-1}=1$, an impossibility.

\medskip\noindent {\bf Case C:} $2m_{r-1}-m_{r+2}-m_{r+3}=0$.
In this case the decreasing order implies that $m_{r-1}=m_{r+2}=m_{r+3}$,
so in fact $m_{r-1}=m_r=m_{r+1}=m_{r+2}=m_{r+3}$; call this value $m$.
We then have $m_i \geq m$ for every $i$,
so we may write  $m_i = m + n_i$ for decreasing non-negative integers $n_i$,
with $n_i = 0$ for $i \geq r-1$.
Now (\ref{equua}) becomes
\[
\sum_{i=1}^{r-2} n_i + 2m = (r+1)/2
\]
forcing $r$ to be odd.

Since $t=0$ we have $d = \sum_{i=1}^{r+1}m_i = (r+1)m+\sum_{i=1}^{r-2} n_i$.
Hence if we define the class $c' = (e;\underline{n})_r$ with $e = \sum_{i=1}^{r-2}n_i$,
then we have $c = mF + c'$,
where $F$ is the anticanonical curve class $(r+1; 1^s)$.
This is the other case of the statement.
\end{proof}

\subsection{The case $s = r+4$}\label{r+4}

With $r+4$ points a parallel approach works up to $\bbP^5$,
at least for the $(-1)$-curve case:

\begin{proposition}\label{-1classr+4}
Suppose $3 \leq r\leq 5$ and $s=r+4$.
Let $C$ be an irreducible curve in $Y^r_s$.
Then $[C]$ is a numerical $(-1)$-class if and only if $C$ is a $(-1)$-Weyl line.
\end{proposition}

\medskip\noindent{\bf Proof:}
It suffices to show that if $c$ is a numerical $(-1)$-class
of the form $d h- \sum_i m_i e_i $ with decreasing multiplicities
representing an irreducible curve
with $d \geq 2$, then $c$ is not Cremona reduced.
Assume by contradiction that $c$ is Cremona reduced;
then Lemma \ref{CRnumjclass} applies and we have

\[
(r-1)(m_{r+2}+m_{r+3}+m_{r+4}) \geq (r-3) + 2\sum_{i=1}^{r+1}m_i.
\]
Now the left hand side is at most $3(r-1)m_{r+1}$
and the right and side is at least $(r-3)+2(r+1)m_{r+1}$, so that
$
3(r-1)m_{r+1} \geq (r-3)+2(r+1)m_{r+1}
$
or $(r-5)m_{r+1} \geq r-3$.

For $r=4,5$ this is an immediate contradiction since $m_{r+1} \geq 0$.
For $r=3$ we must have $-3m_4\geq 0$, and hence $m_4=0$;
then the original inequality
gives $0 \geq 2(m_1+m_2+m_3+m_4)$, forcing all multiplicities equal to zero,
a contradiction.

\rightline{QED}

We cannot extend the above statement for $r > 5$:

\begin{example}\label{counterex} 
Consider $\bbP^6$ with $s=10$ points, and the class $c=3F=(21;3^{10})_6$.
This is a Cremona reduced numerical $(-1)$-class, and hence is not a $(-1)$-Weyl line class.
\end{example}

\begin{remark}
For $s=r+4$ we are in a similar situation:
the only $(-1)$-Weyl lines are the proper transforms of the lines through two points
and the proper transform of the rational normal curve of degree $r$ through $r+3$ of the points.
There are $r+4 \choose 2$ such lines and $r+4$ RNCs.
However, for $r\geq 6$ there are $(-1)$-curves that are not $(-1)$-Weyl classes.
\end{remark}

For dimensions at most $5$ and $s=r+4$ points, we have the same result as for $s=r+3$,
parallel to Proposition \ref{rplus3points}:

\begin{proposition}\label{rplus4points}
Suppose $r\leq 5$ and $s = r+4$.
Let $c \in A^{r-1}$ be a class with positive degree and non-negative multiplicities.
Then $c$ is a $(-1)$-Weyl line class if and only if
\[
\langle c, F \rangle = 3-r \text{ and } \langle c,c \rangle = 3-2r.
\]
\end{proposition}

\medskip\noindent{\bf Proof:}
The proof parallels that of the previous Proposition \ref{rplus3points};
it suffices to prove that a class
with given linear and quadratic invariant is a $(-1)$-Weyl line.
We use the same notation:
let $c = (d;\underline{m})_r$ be such a class, and write
$N_1 = \sum_i m_i$ and $N_2 = \sum_i m_i^2$ as above.

For a fixed degree $d$, the quantity $N_1$ is again fixed to be
$N_1 = ((r+1)d + (r-3))/(r-1)$.
Again $N_2$ would be minimized if all multiplicities were equal (to $m=N_1/(r+4)$), so that
$N_2 \geq \sum_{i=1}^{r+4} m^2 = (r+4)(N_1/(r+4))^2 = N_1^2/(r+4)$.
Therefore
\begin{align*}
3-2r &= d^2-(r-1)N_2 \leq d^2 - (r-1) N_1^2/(r+4) \\
&= d^2 - \frac{r-1}{r+4}(((r+1)d + (r-3))/(r-1))^2 \\
&= d^2 - \frac{1}{(r+4)(r-1)} ((r+1)^2 d^2 + 2(r+1)(r-3) d + (r-3)^2) \\
&= \frac{1}{(r+4)(r-1)} ( ((r+4)(r-1)-(r+1)^2)d^2 - 2(r+1)(r-3) d - (r-3)^2 ) \\
&= \frac{-1}{(r+4)(r-1)} ( (5-r)d^2 + 2(r+1)(r-3) d + (r-3)^2 ).
\end{align*}

For $r=2$ this gives $-1 \leq \frac{-1}{6}(3d^2-6d +1)$
or $3d^2-6d-5 \leq 0$, which forces $d \leq 2$ and leads to the line and the conic.

For $r=3$ this gives $-3 \leq \frac{-1}{14}(2d^2)$
or $d^2 \leq 21$, giving $d \leq 4$.
Now in this case since $-3  = d^2-2N_2$,
the degree cannot be even.  Hence $d = 1$ or $d=3$,
leading again to the line or the twisted cubic which is the RNC.

For $r=4$ the inequality is $-5 \leq \frac{-1}{24}(d^2 + 10d + 1)$,
forcing $d \leq 7$.
However the $\langle c, F\rangle_{r-1}$ value gives us $5d-3N_1 = -1$, so that $d$ must be $1\;\mod 3$.
If $d=7$ then $N_1 = 12$ and $\frac{-1}{24}(d^2 + 10d + 1) = -5$ so that the inequalities are equalities,
and all multiplicities are in fact equal, to $3/2$, which is not possible.
Hence $d=1$ or $d=4$ and we have the line or the RNC again.

Finally when $r=5$ the inequality is $-7 \leq \frac{-1}{36}{24d + 4}$ giving $d \leq 9$,
and $\langle c,F\rangle_{r-1}= -2$ forces $d$ to be odd.

If $d=9$ then the equations give $N_1=14$ and $N_2 = 22$.
Again, for fixed $N_1$, $N_2$ is minimized by having all multiplicities as equal as possible,
and for $r+4=9$ points we have that $N_2$ is minimized with multiplicities $(2^5,1^4)$ with sum equal to $14$.
However for this set of multiplicities $N_2 = 24$, and so $N_2=22$ is impossible.

If $d=7$ then again $N_1=11$ and $N_2=14$; for this $N_1$, $N_2$ is minimized with multiplicities $(2^2,1^7)$
but this gives $N_2 = 15$.

If $d=3$ this implies $N_1=5$ and $N_2 = 4$, a contradiction since $N_2 \geq N_1$ always.

If $d=5$ we have the RNC, and if $d=1$ we have the line.

\rightline{QED}

\begin{proposition}\label{CRnum0classrleq5}
Suppose $r \leq 5$ and $s=r+4$.
Let $c$ be a Cremona reduced numerical $(0)$-class with nonnegative multiplicities.
Then $c$ is either the line class $h-e_1$;
or $r=3$, $s=7$, and $c = F = (4;1^7)$; 
or $r=4$, $s=8$, and $c=2F=(10;2^8)_4$.
\end{proposition}

\medskip\noindent{\bf Proof:}
We start with Lemma \ref{CRnumjclass}.

For $r=5$ and $s=9$ this gives
$4\sum_{i=7}^9 m_i \geq -2 +2\sum_{i=1}^{6}m_i$
which we can re-write as
\[
(m_1-m_7)+(m_2-m_7)+(m_3-m_8)+(m_4-m_8)+(m5-m_9)+(m_6-m_9) \leq 1.
\]
Hence at most one of these non-negative differences is one.
If all are equal to zero, then all multiplicities are equal, say to $m$;
In this case the numerical condition gives that $6d-4\cdot 9m = 2$,
which is impossible.

Hence exactly one of these differences is one, all other five are zero.
We cannot have the second, fourth, or sixth equal to one,
since they are at most the first, third, and fifth respectively.
We conclude that these are zero, which implies that there exists $m$
such that $m_1=m+1$ and $m_i=m$ for $i \geq 2$.
Now the numerical condition gives $6d - 4\cdot (9m+1) = 2$,
implying that $d=6m+1$.
The Cremona reduced conditions then force $m=0$,
and we have the line class $h-e_1$.

For $r=4$ and $s=8$, we have
$3\sum_{i=6}^8 m_i \geq -2 + 2 \sum_{i=1}^5 m_i$
which we re-write as
\[
m_1 + (m_1-m_6)+ 2(m_2-m_6) + 2(m_3-m_7) + (m_4-m_7) + (m_4-m_8) + 2(m_5-m_8)
\leq 2.
\]
Since $m_1 \geq 1$, at most one of the differences here can be positive,
and all of the doubled ones must be zero.
Hence $m_2=m_6$, forcing all $m_i$ for $i\geq 2$ to be equal, say to $m$.
This then reduces to $2m_1-m \leq 2$, leading to having either $m_1=m \leq 2$
or $m_1 = 1$ and $m=0$.

If all $m_i$ are equal to $m$, then the numerical condition gives $5d-3\cdot 8m = 2$,
so we must have $d=10$, $m=2$ since $m \leq 2$.
In the other case we have the line through one point: $h-e_1$.

For $r=3$ and $s=7$ we we have $2\sum_5^7 m_i \geq -2 + 2\sum_1^4 m_i$,
which we may rewrite as
\[
2m_1 + 2(m_2-m_5)+2(m_3-m_6)+2(m_4-m_7) \leq 2
\]
and since $m_1$ is positive each of the differences above must be zero,
and $m_1=1$.
The two cases are then when $m_2=0$ (giving the line class $h-e_1$)
or when $m_2=1$
(which gives the $F$ class).

For $r=2$ and $s=6$ we have $\sum_4^6 m_i \geq -2 + 2\sum_1^3 m_i$
which also leads only to the line class $h-e_1$.

\rightline{QED}

We note that the exceptional case $c = (10;2^8)$ in $\bbP^4$ has
$\langle c,c\rangle = {10}^2-3\cdot 8 \cdot 2^2 = 4$,
and the other exceptional case $c = (4;1^7)$ in $\bbP^3$ has
$\langle c,c\rangle = {4}^2-2\cdot 7 \cdot 1^2 = 2$,
both of which are different from the $(0)$-Weyl line class values.
(which are $-2$ and $-1$, respectively).
Hence we have:

\begin{corollary}
\label{num0classNot2F}
Suppose $r \leq 5$ and $s=r+4$; a Cremona reduced class $c$ with nonnegative multiplicities is a $(0)$-Weyl line if and only if
$\langle c,c\rangle = 2-r$ and $\langle c,F\rangle = 2$.
\end{corollary}

In this case of $r+4$ points, for $r\geq 6$,
it is no longer the case that the two invariants
$ \langle c, F \rangle$ and  $\langle c,c \rangle$
pick out the Weyl classes.

\begin{example}\label{counterex2}
Consider again the class $c = 3F = (21;3^{10})_6$
of curves of degree 21 in $\bbP^6$ with ten points of multiplicity $3$.
We have $ \langle c, F \rangle= 7\cdot 21 - 5(10\cdot 3) =-3$ and
$\langle c,c \rangle_{r-1}=21^2-5*10*9=-9$
as is the case for the line through two points,
but this is not a $(-1)$ Weyl class.
In fact, the class $c$ is Cremona reduced, as well.
\end{example}

It is also true that there are counterexamples in $\bbP^3$, for larger number of points.
The class $c= (7;4,1^{10})$ has
$\langle c,F\rangle_{r-1} = 0$
and
$\langle c,c\rangle_{r-1} = -3$
as a $(-1)$-Weyl line does;
however it is not in the Cremona orbit of $h-e_1-e_2$.

\subsection{The proof of Theorem \ref{class} and Conclusions}

Using Proposition \ref{rplus3points} and Proposition \ref{rplus4points}, we have:

\begin{corollary}\label{MDSminus1}
If $r\geq 3$ and $Y=Y^r_s$ is a Mori Dream Space,
(i.e., $r=3,4$ and $s\leq r+4$ or $r \geq 5$ and $s \leq  r+3$)
the only classes $c\in A^{r-1}(Y)$ 
with positive degree and non-negative multiplicities
that satisfy the equations
$$ \langle c, F \rangle= 3-r \text{ and }  \langle c, c \rangle= 3-2r$$
are either the proper transform of a line through two points or the rational normal curve of degree $r$ through $r+3$ points.
\end{corollary}


\begin{proof}[Proof of Theorem \ref{class}] Assume $Y$ is $Y^5_9$ or is a Mori Dream Space.
The first part of the previous proof applies as well here.
For $(-1)$-curves the equivalence of conditions $(3)$ and $(2)$ in Theorem \ref{class}
follows again from Corollary \ref{MDSminus1}. 
 If $r=2$ the proof of Theorem  \ref{class} for the del Pezzo surfaces follows from Theorem \ref{Y2sirreducible}, Proposition \ref{planar} and Proposition \ref{del}.
 In higher dimensions, $r>2$, Proposition \ref{-1classr+4} proves that $(4)$ implies $(2)$, proving the equivalences $(1)$, $(2)$, and $(4)$. 
\end{proof}

\begin{corollary}\label{main theorem}
If $Y$ is Mori Dream Space and $r\geq 3$ or $Y=Y^5_9$
then the only $(-1)$-curves are the ones described in Example \ref{line}.
\end{corollary}

\begin{theorem}\label{class3}
Theorem \ref{class2} for $i=1$ 
holds  if $r$ is even.
\end{theorem}
Indeed, the equivalence  between  numerical $(1)$-classes and  $(1)$-Weyl line  in even dimensional spaces $Y^r_{r+3}$ follows from  Proposition \ref{icremona} and this completes the cycle.

\begin{remark}\label{rem} Theorem \ref{class} holds for irreducible $(0)$-curves in Mori Dream cases or $Y^5_9$ with two exceptions. The proof follows from  Proposition \ref{CRnum0classrleq5}; one exception is the anticanonical class $F$ in $Y^3_7$ that contains a $(0)$-curve that is not a $(0)$-Weyl line. The only other candidate for a $(0)$-curve that is not a $(0)$-Weyl line in the above hypothesis is $r=4$ and $s=8$ and $2F$.

 Similarly, Theorem \ref{class3} does not hold if $r$ is not even. Indeed, consider $i=1$ and $r=3$, and $s=6$ then $F$ is a $(1)$-curve that is not a $(1)$-Weyl line in $Y^3_6$ (as in Example \ref{F class in 3}) .
\end{remark}

Next proposition also holds for $(0)$-curves and $vdim (C)=r-1$ with part (3) including the $(0)$-curves in $F$ in $Y^3_7$ and $2F$ in $Y^4_8$.

\begin{proposition} \begin{enumerate}
\item Any $(-1)$-curve has $vdim (C)=0$.
\item If $Y$ is Mori Dream Space or $Y^5_9$, the only irreducible curve classes with $vdim (C)=0$ are classes of $(-1)$-curves and $-K_{Y^2_8}$.
\end{enumerate}
\end{proposition}

\begin{proof}
Part (1) follows from the observation that $vdim(C)=(-K_Y \cdot C) + (r-3)$
(equation \ref{param}) is stable under the Weyl group action
since the anti-canonical divisor is also stable.
Part (2) follows from Theorem \ref{class} since the condition $vdim(C)=0$ is equivalent to
$C$ being a $(-1)$-numerical class or the anti-canonical class in $Y^2_8$.
\end{proof}

We ask the following question:
\begin{question}
A smooth rational irreducible non-degenerate curve in $Y^r_s$ is rigid if and only if $vdim(C)=0$.
\end{question}

We do have the following for the Mori Dream Space cases now:

\begin{proposition}\label{fmds}
A Mori Dream space has finitely many $(i)$-curves for $i \in \{-1,0,1\}$.
\end{proposition}

\begin{proof}
If $i= -1$ the statement follows from Theorem \ref{class}.
Remark \ref{rem} implies that $(0)$-curves in a Mori Dream Space are either $(0)$-Weyl lines
 or $F$ in $Y^{3}_{7}$ or $2F$ in $Y^{4}_{8}$. 
Assume now $C$ is a $(1)$-curve.
If $Y$ is a del Pezzo surface then a $(1)$-curve is in the Weyl orbit of a line
and we conclude since the Weyl group is finite by Corollary \ref{mds}.
If $s\leq r+3$, Proposition \ref{cremona} implies that a numerical $(1)$-class is either a $(1)$-Weyl line or is at the form $C$ is a Weyl orbit of the curve $mF+C'$ where $m$ and the degree of $C'$ is bounded above by $\frac{r+1}{2}$. Since the Weyl group is finite we have finitely many possibilities for a $(1)$-curve if $s\leq r+3$. If $Y=Y^4_8$ a numerical $(1)$-class satisfies the equality $5d-3\sum_{i=1}^{8}m_i=5$. Suppose the class is Cremona reduced ie $d\geq \sum_{i=1}^{5} m_i$, then all multiplicities are bounded above by $5$, this implies that the degree $d \leq 25$. We conclude again since the Weyl group is finite. Similarly, if $Y=Y^3_7$ we obtain that all multiplicities $m_i\leq 2$, while the linear invariant gives that $d \leq 8$, and we conclude again via the finiteness of the Weyl group.

\end{proof}

\begin{example}\label{F class in 3}[ The $F$-class in $Y^3_8$] 
In order to construct examples of $(-1)-$curves 
that are not $(-1)-$ Weyl lines 
we study the $F$-class in $Y^3_8$. 
The $F$-class in $Y^3_8$ represents smooth rational curves 
and is also the class of an elliptic curve.
\end{example}
An interesting example of the type of phenomena that can appear
is afforded by the $F$ class in $\bbP^3$ with $s=8$,
which is the class of quartics passing simply through $8$ points.
This class is Cremona reduced, and as we have noted, is invariant under the Weyl group.
It is a numerical $(-1)$-class: $(-K, F) = 0 = 3-r$ in this case.
It is the class of an effective curve: the easiest way to see this is to note that if you break up the $8$ points into pairs, then the four lines joining the four pairs is a (disconnected, reducible) curve in $\bbP^3$ in this class.

It is also the class of an irreducible curve of genus one.
The linear system of quadrics in $\bbP^3$ has dimension $9$
and the quadrics through the $8$ points is a pencil $\mathcal{P}$.
The base locus of this pencil, which is the intersection of any two of the quadrics in the pencil,
is (in general) a smooth curve $E$ of degree four and genus one through the eight points.
If one chooses a smooth quadric in the pencil,
and considers it as being isomorphic to $\bbP^1\times\bbP^1$,
then the genus one curve has bidegree $(2,2)$.

Since the $8$ points are general,
the general quadrics in the pencil will be smooth
and exactly four members will be cones over smooth conics;
none of the four vertices are among the $8$ given points.
In each case there are lines on the quadrics;
in the smooth case there are the two rulings,
and in the cone case there is the one system of lines through the vertex.
Consider the incidence correspondence
$\calI = \{(Q,R)\;|\; Q \in \calP, R = \text{ a ruling on Q}\}$.
We note that $\calI$ is in $1$-$1$ correspondence with the set of $g^1_2$'s on the base curve $E$;
a ruling on one of the quadrics restricts to $E$ as a $g^1_2$
and inversely a $g^1_2$ gives a quadric in the pencil as the union of the secants,
with the secants forming the ruling.

For a pair $(Q,R)\in \calI$, if $Q$ is smooth,
we may consider the linear system $3R+R'$
where $R'$ is the other ruling;
this system is of bidegree $(3,1)$ and has dimension $7$;
hence there is a unique member $C \in 3R+R'$
passing through the first $7$ of the given original base points$ \{p_i\}$.
We ask that, as we vary $Q$ in the pencil,
that this curve $C$ pass also through the last (eighth) point.

Note that $3R+R' \equiv 2R+H$ where $H$ is the hyperplane class on $Q$,
and that these systems therefore restrict to $2g^1_2+H|_E$ on the genus one curve $E$.
The entire construction depends on the choice of the $g^1_2$ on $E$,
and we are therefore asking how many $g^1_2$'s are such that
$|2g^1_2+H|_E - (\sum_{i=1}^7p_i)|$ contains the point $p_8$.

We may always write a $g^1_2$ as $p_8+p$ for a variable point $p \in E$.
This condition is then that $2p+2p_8+H|_E-(\sum_{i=1}^7p_i)) =p_8$ in the Jacobian of $E$;
there are four such points $p$ of course (any solution plus the two-torsion points).
This gives four curves $C$ of bidegree $(1,3)$ on a quadric through all $8$ points.
Each of these is a smooth rational curve, and is indeed a $(-1)$-curve in $Y^3_8$
with class $F$.
(The normal bundle is $\calO(-1)\oplus\calO(-1)$ using results of \cite{EV}.)

\section{Infinity of $(i)$-curves}\label{infinite}

In this section we will prove the following.

\begin{theorem}\label{nms2}
There are finitely many classes of $(0)$-curves or equivalently finitely many classes of $(1)$-curves if and only if 
$Y^r_s$ is a Mori Dream Space.
\end{theorem}

The proof is presented after Corollary \ref{infinite01rgeq5} in the first subsection below.

Although $Y^5_9$ is not a Mori Dream Space we can determine all $(-1)$- and $(0)$-curves.
\begin{proposition} 
\begin{enumerate}
\item In $Y^5_9$ all $(0)$-curves are $(0)$-Weyl lines (or have numerical $(0)$-classes).
\item In  $Y^5_9$ the only $(-1)$-curves are ones in Example \ref{line}
(i.e.  the proper transforms of lines through two points and the proper transform of the rational normal curves through $n+3$ points).
\end{enumerate}
\end{proposition}
\begin{proof}
 Indeed Theorem \ref{class} implies that there are {\it finitely} many $(-1)$-curves;
 therefore the notion of $(-1)$-curves do not detect the finite generation of the Cox Ring of $Y$. Moreover, Proposition \ref{CRnum0classrleq5} implies that all Cremona reduced $(0)$-curves are lines through two points.
 Therefore every $(0)$-curve is a $(0)$-Weyl line (and there are infinitely many),
 see  Remark \ref{rem}.
\end{proof}

Can movable curves be used in the study of Mori Dream Spaces?
\begin{question} Does Theorem \ref{nms2} hold for other varieties?
\end{question}

\subsection{Infinity of movable curves in $\mathbb{P}^r$}\label{infinite1}

In this Section we analyze $(i)$- Weyl lines in  $Y^r_s$ that is not a Mori Dream Space: namely in \ref{infinite1} we begin with an analysis of the $(1)$-Weyl lines, which are the Cremona images of the general line $h$ in $\bbP^r$, and in \ref{infinite2} we study the rigid curves, $i=-1$.

We will analyze the process of applying the Cremona to the $r+1$ lowest multiplicity points, 
and then re-ordering so that the multiplicities are in ascending order,
in order to simplify the notation. We begin with the $s=r+4$ case. 
It is more useful for this section to place the multiplicities in ascending order,
and we will systematically do that here.


\begin{lemma}\label{recursionCremona} 
Let $L = (d;m_1 \leq m_2 \leq \ldots \leq m_{r+4})$
be a curve class in $A^{r-1}(Y^r_{r+4})$
with positive degree and non-negative multiplicities in nondecreasing order.
Let $I \subset \{1,\ldots,r+4\}$ have size $r+1$, so that there are three indices missing;
assume that those indices are $1, k, \ell$ with $4 \leq k$ and $k+3 \leq \ell \leq r+2$.
(This implies in particular that $1 \notin I$; $2, 3 \in I$;
$m_{r+3}, m_{r+4} \in I$; and $r \geq 5$.)
Assume that
\begin{equation}\label{recursioninequality}
d > \sum_{i \in I} m_i, 
\text{ or } d = \sum_{i \in I} m_i \text{ and } m_{r+4} > m_1.
\end{equation}
Then the Cremona image $\phi(L)$
(where $\phi$ is based at the first $r+1$ points, those with the lowest multiplicities)
has the form $\phi(L) = (d’; \{m_i’\})$
(where the $m_i’$ are also placed in increasing order)
and these parameters also satisfy (\ref{recursioninequality}).
Moreover in this case $d' > d$.
\end{lemma}

\medskip\noindent{\bf Proof:}
We define $t = d-\sum_{i=1}^{r+1}m_i$,
and note that the degree of $\phi(L)$ is then $d’ = d+(r-1)t$
and the multiplicities are 
\[
m_1+t, m_2+t,\ldots,m_{r+1}+t, m_{r+2}, m_{r+3}, m_{r+4}
\]
in some order, by Proposition \ref{CremonaAY}(a).

We first claim that $m_{r+4} \leq m_1+t$.
This is equivalent to having $d \geq \sum_{i=2}^{r+1}m_i + m_{r+4}$,
which is implied by (\ref{recursioninequality}) and the assumptions that $1 \notin I$
and $m_{r+4} \in I$.

Since this is true, the re-ordering of the multiplicities of $\phi(L)$
to be in increasing order gives
\[
m_1’ = m_{r+2}, m_2’ = m_{r+3}, m_3’ = m_{r+4},
\text{ and for } 4 \leq i \leq r+4, m_i’ = m_{i-3}+t.
\]
Therefore (\ref{recursioninequality}) for $\phi(L)$, which is 
$d’ > \;(\text{ resp. }\geq)\;\sum_{i\in I} m_i’$,
is equivalent to
\[
d+(r-1)t > \;(\text{ resp. }\geq)\; m_{r+3} + m_{r+4} + \sum_{i \in I, i \geq 4} (m_{i-3}+t)
\]
since $1 \notin I$ and $2,3 \in I$.
Note that since $|I|=r+1$, the sum above contains $r-1$ indices;
hence we may subtract $(r-1)t$ from both sides of this inequality to obtain
\[
d > \;(\text{ resp. }\geq)\; \sum_{i \in I, i \geq 4} m_{i-3} + m_{r+3} + m_{r+4}.
\]
To prove this, it suffices to show that the right side of (\ref{recursioninequality})
is at least the right side of the above, i.e., that
\[
\sum_{i \in I} m_i \geq \sum_{i \in I, i \geq 4} m_{i-3} + m_{r+3} + m_{r+4}.
\]
Now the sum of $r+1$ multiplicities on the right side of this
are exactly those with indices in the set $I' = \{1,\ldots,r+4\} - \{k-3, \ell-3, r+2\}$.
If we denote the $j$'th index in $I$ by $I(j)$, and similarly for $I'$,
it will suffice to show that $I(j) \geq I'(j)$ for each $j$.

The set of indices in $I$ increase by one at the $1, k, \ell$ points;
those of $I'$ increase at the $k-3, \ell-3, r+2$ points.
Since $1 \leq k-3$, the first increase of $I$ is no later than that of $I'$;
since $\ell \leq r+2$, the third increase of $I$ is no later than that of $I'$.
The only failure then would be if the second increase of $I$ would be later than the second of $I'$, and that will only happen if $\ell-3 < k$.
This is forbidden by the assumption that $k+3 \leq \ell$.

This proves that (\ref{recursioninequality}) holds for $\phi(L)$ also.

To finish and show $d' > d$, we must show that $t >0$, or that
$d > \sum_{i=1}^{r+1} m_i$.  This follows immediately from (\ref{recursioninequality}) if the inequality is strict; if not, but $m_{r+4} > m_1$, we get the result since $r+4 \in I$.
Finally we must show that if the inequality
(of the parameters for $\phi(L)$ and hence also for $L$) is an equality,
then $m_{r+4}' > m_{1}'$, i.e., that $m_{r+1}+t > m_{r+2}$;
this is equivalent to $d > \sum_{i=1}^r m_i + m_{r+2}$.
The indices in $I$ dominate those in this sum, and so the only way this could fail is if
we have equality here, and all $m_j$'s are equal.
This is a contradiction, since $m_{r+4} > m_1$.

\rightline{QED}

In case of $r = 3, 4$, a completely parallel lemma can be proved, but we need to have one more point ($s=r+5$).  In particular, using the same notation, if
\begin{equation}\label{recursioninequalityr=34}
d > m_3+m_4+\sum_{i=7}^{r+5} m_i, 
\text{ or } d = m_3+m_4+\sum_{i=7}^{r+5} m_i \text{ and } m_{r+5} > m_1
\end{equation}
then we have the same conclusion:
the Cremona image $\phi(L)$ has parameters also satisfying 
(\ref{recursioninequalityr=34}) and the degree increases.
We leave it to the reader to check the details,
which are parallel in all respects to those of Lemma \ref{recursionCremona}.

For $r=2$ and $s=9$, the recursive argument requires the inequality
\begin{equation}\label{recursioninequalityr=2}
d > m_3+m_6+m_9, 
\text{ or } d = m_3+m_6+m_9 \text{ and } m_9 > m_1
\end{equation}
and again the same argument goes through, with the same conclusion.


\begin{corollary}\label{infinite01rgeq5}
If $r \geq 5$ and $s \geq r+4$,
or $r \geq 3$ and $s \geq r+5$,
or $r=2$ and $s \geq 9$, there are infinitely many $(1)$-Weyl line classes
and $(0)$-Weyl line classes.
\end{corollary}

\medskip\noindent{\bf Proof:}
For $r \geq 5$, it suffices to prove this for $s = r+4$.
For the $(1)$-Weyl line classes, apply Lemma \ref{recursionCremona} repeatedly starting with $L=h$, the general line class.
Since the degrees increase without bound
(using the strict inequality of (\ref{recursioninequality})), we have the result.
For the $(0)$-Weyl line classes, we apply the same lemma
starting with $L=h-e_1$, which using the $(d;\underline{m})$ notation
is given by $(1;0^{r+3},1)$.
Here we do not have the strict inequality in (\ref{recursioninequality}), but we do have $m_{r+4}=1 > m_1=0$; the result follows.

For $r = 3,4$, it suffices to prove this for $s = r+5$;
the same argument holds, this time using the recursive assumption of (\ref{recursioninequalityr=34}).

Finally for $r=2$, it suffices to prove this for $s=9$, and again this is done via
the recursion supplied by (\ref{recursioninequalityr=2}).

\rightline{QED}

We can now prove the statement at the beginning of the Section:

\medskip\noindent{\bf Proof of Theorem \ref{nms2}:}
One direction follows from Proposition \ref{fmds}. The other direction follows from Corollary \ref{infinite01rgeq5}.

\rightline{QED}

One can also be very explicit about the recursion that produces the infinite family of curves in these cases with low $r$.
For example, for $r=3$ and $s=8$, one can easily see (by induction) that the sequence of $(0)$-Weyl line classes $\phi^{i}(h-e_8)$ has the form $d = i^2+i+1$, and multiplicities
$((k^2)^4, (k^2+k)^3, (k^2+k+1)^1)$ if $i=2k$;
$((k^2+k)^3, (k^2+k+1)^1,(k^2+2k+1)^4)$ if $i=2k+1$.

For $r=2$ and $s=9$, 
the sequence of $(0)$-Weyl line classes $\phi^{i}(h-e_9)$ has the form
$d = 1+i(i+1)/2$, and multiplicities
$((k+3k(k-1)/2)^3, (2k+3k(k-1)/2)^3, (3k(k+1)/2)^2, (1+3k(k+1)/2)^1)$ if $i=3k$;
$((3k(k+1)/2)^2, (1+3k(k+1)/2)^1, ((k+1)(1+3k/2))^3, (2k+3k(k-1)/2)^3)$ if $i=3k+1$;
$((3k(k+1)/2)^2, (1+3k(k+1)/2)^1, ((k+1)(1+3k/2))^3, ((k+1)(2+3k/2))^3)$ if $i=3k+2$.


We now observe that the results of Section \ref{-1curvesMDS} imply that,
for the Mori Dream Spaces $Y^r_s$
($r=2, s \leq 8; r=3,4, s \leq r+4; r \geq 5, s \leq r+3$)
there are only finitely many numerical $(1)$-classes
and numerical $(0)$-classes.
Hence this gives a criterion for these spaces $Y^r_s$ being Mori Dream Spaces:

\begin{theorem}\label{finiteclass}
The space $Y^r_s$ is a Mori Dream Space
if and only if there are finitely many numerical $(1)$-classes in $A^{r-1}$,
and if and only if there are finitely many numerical $(0)$-classes in $A^{r-1}$.
\end{theorem}

We make an additional observation here, concerning these infinite series of $(0)$- and $(1)$-Weyl classes.  Each imposes a condition on divisors to be effective: for any such class $c$, it is a necessary condition for $D$ to be effective, that $(D\cdot c) \geq 0$.  We note below that these conditions are independent:

\begin{proposition}\label{independentconditions}
For the infinite series $\{c_k\}$ of $(0)$-Weyl classes constructed above,
we have that $c_{k+1}$ is not in the convex hull of $\{c_j \}_{j\leq k}$ for all $k$.
\end{proposition}

\begin{proof}
We argue by contradiction: suppose that $c_{k+1} = \sum_{j=1}^k a_j c_j$ for nonnegative real numbers $a_j$.
Since each class $c_j$ is an $(0)$-Weyl class,
we have that $\langle c_j , F \rangle = 2$ for every $j$;
applying this to $c_{k+1}$ and dividing by $2$ gives that
$1 = \sum_{j=1}^k a_j$.
This would then imply that $\deg(c_{k+1}) = \sum_{j=1}^k a_j \deg(c_j)$.
However this is not possible, since the degrees of the classes $c_j$ increase monotonically.
\end{proof}

We remark that the same proof applies for $(1)$-Weyl classes unless $\langle h,F\rangle = 0$,
which only happens if $r=3$.


\subsection{Infinity of rigid curves in $\mathbb{P}^r$}\label{infinite2}

In this Section we discuss the question of infinite number of $(-1)$-curves in the sense of Kontsevich \cite{maxim1}, in $Y^r_s$, where $s \geq r+5$ via
Corollary \ref{infrigid}. 
For $r=3, 4$, the Weyl group with $r+4=7,8$ points is finite,
and so there are only finitely many $(-1)$-Weyl lines.
However we can prove a similar statement for $r+5$ (or more) points.

\begin{lemma}\label{infiniteWeylLemmar4} Assume $r \geq 3$.
Let $L = (d;m_1 \leq m_1 \leq \ldots \leq m_{r+5})$
be a curve class in $A^{r-1}(Y_{r+5}^r)$
with positive degree and non-negative multiplicities.
Assume that
\begin{equation}\label{infinite2inequality}
d > m_3 + m_4 + \sum_{i=7}^{r+5} m_i.
\end{equation}
Then the Cremona image $\phi(L)$
(where $\phi$ is based at the first $r+1$ points, those with the lowest multiplicities)
has the form $\phi(L) = (d’; \{m_i’\})$
(where the $m_i’$ are also placed in increasing order)
and these parameters also satisfy (\ref{infinite2inequality}).
Moreover in this case $d' > d$.
\end{lemma}

\medskip\noindent{\bf Proof:}
We note that (\ref{infinite2inequality}) implies that
\begin{equation}\label{dbiggerms}
d > m_{i_1}+m_{i_2}+\ldots + m_{i_{r+1}}
\;\;\;\text{if}\;\;\;
i_1 \leq 3, i_2\leq 4, \;\text{and}\; i_j \leq j+4 \;\text{for}\; 3\leq j \leq r+1.
\end{equation}
We again define $t = d-\sum_{i=1}^{r+1}m_i$,
and note that the degree of $\phi(L)$ is then $d’ = d+(r-1)t$
and the multiplicities are 
\[
m_1+t, m_2+t,\ldots,m_{r+1}+t, m_{r+2}, m_{r+3}, m_{r+4}, m_{r+5}
\]
in some order, by Proposition \ref{CremonaAY}(a).

We claim that $m_{r+5} \leq m_1+t$.
This is equivalent to having $d \geq \sum_{i=2}^{r+1}m_i + m_{r+5}$,
which is implied by (\ref{dbiggerms}).

Since this is true, the re-ordering of the multiplicities of $\phi(L)$ to be in increasing order gives
\[
m_1’ = m_{r+2}, m_2’ = m_{r+3}, m_3’ = m_{r+4}, m_4'=m_{r+5},
\text{ and for } 5 \leq i \leq r+5, m_i’ = m_{i-4}+t.
\]

We next check (\ref{infinite2inequality}) for $\phi(L)$: this is
$d' > m_3' + m_4' + m_7' + \cdots + m_{r+5}'$
which is equivalent to
\begin{align*}
d+(r-1)t &>  m_{r+4}+m_{r+5}+(m_3+t)+\cdots+(m_{r+1}+t) \\
&= m_3+\ldots +m_{r+1}+m_{r+4}+m_{r+5} +(r-1)t
\end{align*}
which is equivalent to having $d>m_3+\ldots+m_{r+1}+m_{r+4}+m_{r+5}$.
This follows from (\ref{dbiggerms}) if $r \geq 3$.

Finally $d'>d$ if and only if $t > 0$, which also follows from (\ref{dbiggerms}).

\rightline{QED}

 We will now apply Lemma \ref{infiniteWeylLemmar4} to $(-1)$-Weyl lines,
giving a different perspective on the topic from Corollary \ref{infinite-1curves}:

\begin{corollary}\label{infrigid}
For $r\geq 3$, with $r+5$ or more points in $\bbP^r$,
the orbit of the proper transform of a line through two points is infinite.
Hence there is an infinite number of $(-1)$-Weyl lines and numerical $(-1)$-classes.
\end{corollary}

\medskip\noindent{\bf Proof:}
We will use exponents to denote repeated multiplicities.
Using the notation above, we start with the line class $L_1 = (1;0^{r+3},1^2)$ and iterate,
and set $L_{i+1}=\phi(L_i)$.
We cannot immediately apply the Lemma above
because the inequality of the Lemma is not satisfied for $L_1$.
We have $t=1$ so $L_2 = (r;0^2,1^{r+3})$;
this also does not satisfy (\ref{infinite2inequality}) of the Lemma,
so we continue, noting that $t=1$ again.
Hence $L_3=(2r-1;1^2,2^{r-1},1^4) = (2r-1;1^6,2^{r-1})$;
again the inequality fails, and now $t=3$.
Continuing, if $r\geq 5$ we have
$L_4=(5r-4; 4^6, 5^{r-5}, 2^{4}) = (5r-4; 2^4,4^6, 5^{r-5} )$
and this does satisfy the inequality;
hence so does $L_i$ for all $i \geq 4$.
 Since the degrees strictly increase,
they increase without bound and the orbit is infinite.

If $r=4$ then $L_4=(13;1,2,2,2,3,3,3,3,3)$,
$L_5=(22;3,3,3,3,4,5,5,5,6)$, and
$L_6=(40;5,5,5,6,9,9,9,10)$.
The class $L_6$ does satisfy the conditions of the Lemma,
and therefore so does $L_i$ for all $i\geq 6$.
Since the degrees strictly increase,
they increase without bound and the orbit is again infinite.

For $r=3$ the Lemma is not useful; the inequality is never satisfied.
The reader can check though that if we write $i = 2j+e$ with $e\in\{0,1\}$, then
$L_i = (2i-1;(j+e-1)^2,j^4,(j+e)^2)$ 
(for each $L_i$ we have $h=2i-1-2(j+e-1)-2j = 1$)
which gives an infinite set also.

\rightline{QED}

We remark that the infinity of $(-1)$-Weyl lines does not directly follow from the fact that the Weyl group is infinite: $Y^5_9$ is not a MDS and has an infinite Weyl group
but has finitely many (-1)$-$curves,
so there is a subtlety here.

\section{Applications}\label{apl}
A \emph{movable} curve in $Y^r_s$ is a curve that moves in a large enough family so that a general point is contained in at least one member of the family.

In Section \ref{cone} we give applications of the theory of movable curves 
to the faces of the effective cone of divisors, 
while in Section \ref{baselocus} we give applications of the theory of rigid curves 
to the dimensionality problems of linear systems with multiple points in $\mathbb{P}^r$, 
resolutions of singularities and we conclude with some examples of vanishing theorems. 
The most important results of this section are contained in Theorem \ref{nms1} (presented directly below), 
Theorem \ref{effective cone} and Theorem \ref{nms} 
that prove Theorem \ref{extremal rays}. 
Theorem \ref{effective cone}
gives a one-to-one correspondence between faces of this cone
and the collection of $(0)$- and $(1)$-Weyl lines. 
More precisely, a divisor $D\in Pic(Y^r_{r+3})$ is effective 
if and only if $D \cdot C \geq 0$ for every curve $C$ that is a $(0)$- or a $(1)$-Weyl line on $Y$.
In general, if $Y$ is not a Mori Dream Space (arbitrary $s$),
then Theorem \ref{nms} implies that
$(0)$- and $(1)$-Weyl lines define an infinite set of conditions for effectivity of a divisor.

We now present the proof of Conjecture \ref{HardConjecture} for $i =0,1$ which we use in the sequel.

\begin{proposition}\label{ConjectureProofTheorem}
Conjecture \ref{HardConjecture} is true for $i = 0$ and $i=1$:
for all $r$, $s$, every $(i)$-Weyl line is an $(i)$-curve.
\end{proposition}

\medskip\noindent{\bf Proof:}
The proof of (a) is relatively straightforward, after we make the following observation.
First, the statement is true for the actual line through $1$ or $0$ points (the initial case of a $(0)$-Weyl line or a $(1)$-Weyl line): the normal bundle of a line is a direct sum of $\calO(1)$'s, and upon blowing up, the normal bundle is twisted by $\calO(-1)$.

The argument proceeds by induction on the number of standard Cremona transformations required to arrive at the given $(i)$-Weyl line.  When that number is zero, we have the initial case above.  When that number is one, the curve is always a rational normal curve, and has the expected balanced normal bundle too.

If $C$ is an $(i)$-Weyl line that is obtained by applying $k$ standard Cremona transformations to a line, then we may write $C = \phi(C')$ where $C'$ is obtained by applying $k-1$ such transformations to a line, and $\phi$ is a standard Cremona transformation. By induction, $C'$ will be an $(i)$-curve in $Y^r_s$.

The standard Cremona transformation $\phi$ on $Y^r_s$,
as explained in Section \ref{standardCremona},
is factored by systematically blowing up and down the proper transforms of the linear subspaces of codimension at least two spanned by subsets of the $r+1$ initial base points.  Hence if these proper transforms are disjoint from the curve $C'$, the curve $C$ will be isomorphic to $C'$, and will have the same normal bundle, proving that $C$ is also an $(i)$-curve in the transformed $Y^r_s$.  It suffices to show that $C'$ is disjoint from the proper transforms of the codimension two subspaces, since these contain the others.

Let $\psi$ be the composition of standard Cremona transformations 
that take $C'$ back to a line.  
To show that $C'$ is disjoint from 
the finite number of proper transforms 
of the codimension two subspaces $\{L_\alpha\}$ in question, 
it suffices to show that $\psi(C')$ is disjoint from the transforms $\psi(L_\alpha)$,
which are a finite number of codimension two subvarieties in $Y^r_s$.

This is obvious for $i=1$: the line $\psi(C')$ is a general line, and can be chosen to be disjoint from any finite set of codimension two subvarieties.
For $i=0$, the line $\psi(C')$ is a line through one of the $s$ points, but is otherwise general;
in this case a simple dimension count shows that the general member of the $(r-1)$-dimensional family of such lines can be chosen to be disjoint from the finite set of codimension two subvarieties as well.



\rightline{QED}

\begin{corollary}\label{infinite-1curves}
For $r \geq 3$ and $s \geq r+6$, or $r \geq 5$ and $s \geq r+5$,
there are infinitely many $(-1)$-curves in $Y^r_s$.
\end{corollary}

\begin{proof}
We know from the above that there are infinitely many $(0)$-Weyl line classes in these cases,
and using Proposition \ref{ConjectureProofTheorem}, 
we see that these are all $(0)$-curves.
By imposing an additional point of multiplicity one, 
we will create infinitely many $(-1)$-Weyl line classes, 
all of which are $(-1)$-curves.
\end{proof}


The following Theorem \ref{nms1} was first proved by Mukai in \cite{Mu05} 
using the infinity of $(-1)$ Weyl divisors. 
In this section we will prove this result using the theory of movable curves that we introduced in this paper.


\begin{theorem}\label{nms1}
If $F^2=\langle F, F \rangle \leq 0$,
(i.e., $r=2$ and $s\geq 9$; or $r=3, 4$ and $s\geq r+5$; or $r\geq 5$ and $s\geq r+4$)
then $Y^r_s$ is not a Mori Dream Space.
\end{theorem}

\medskip\noindent{\bf Proof:}
If $F^2 \leq 0$, Corollary \ref{infinite01rgeq5} implies that
there are infinitely many classes of $(0)$-Weyl lines.
Each such curve class gives a facet of the cone of effective divisors of $Y^r_s$
(see Theorem \ref{nms}). Effective cone facets given by $(0)$-classes are independent, via Proposition \ref{independentconditions}.
We conclude that $Y$ can not be Mori Dream space
because its effective cone is not rational polyhedral.

\rightline{QED}

Question \ref{nms3} asks whether the extremal rays of the cone of movable curves 
in all Mori Dream spaces $Y^r_{s}$ are $(0)$-Weyl lines and $(1)$-Weyl lines.
\begin{remark}
In $Y^r_{r+3}$, Theorem \ref{class2} reveals that $(0)$-curves are $(0)$-Weyl lines
and in even dimensions Theorem \ref{class3} shows that $(1)$-curves are $(1)$-Weyl lines. 
Remark \ref{rem} shows that 
the $F$-class in $Y^3_6$ contains a $(1)$-curve that is not $(1)$-Weyl line. 
However this $(1)$-curve is not an extremal ray because is the sum of two $(0)$-Weyl lines in Example \ref{line}.
\end{remark}

\subsection{Effective cone of divisors}\label{cone}
In this section we discuss the theory of movable curves
and applications to the cone of effective divisors when $s\leq r+3$.
There is a vast literature studying the geometry of the space $Y^r_{r+3}$
and for a more proper list of citations we will refer to \cite[Section 0]{DP}.
The chamber decomposition of the effective divisorial cone of $Y^r_{r+3}$
is exposed in \cite{AC}.

Also Mukai proves that $Y^4_8$ is isomorphic to
the moduli space $S$ of rank two torsion free sheaves $G$,
with prescribed Chern classes $c_1(G)= - K_Y$ and $c_2(G)=2$.
In \cite{Cas2} the authors use the work of Mukai and Gale duality
that relates spaces $Y^4_8$,$ Y^2_8$, and $S$
in order to study the effective cone of divisors for $Y^4_8$.
In particular, Gale duality gives a correspondence between extremal rays for the effective cone divisors on $Y^4_8$ 
and curves $C$ in $Y^2_8$ for which $C \cdot C=0$ and $C \cdot K_Y= -2$.
These curves are $(0)$-curves on $Y^2_8$
and we prove they correspond to $(0)$-Weyl lines.
The correspondence is more general as Corollary \eqref{div}, Section \ref{section planar}
proves that $(0)$-divisorial classes are equivalent to $(0)$-Weyl hyperplanes on $Y^2_s$. This observations motivate the following considerations.

In this section we will prove that in $Y=Y^r_{r+3}$
the collection of $(0)$- and $(1)$-curves, $\mathcal{C}$,
give all faces of the effective cone of divisors.
In order to see this,
let us review known results about the faces of effective cones.
We recall in these cases the movable cone of divisors
(ie divisor classes that do not contain divisorial base components)
consists of effective divisors
that have positive intersection with all other effective divisors of $Y$
with respect to the Dolgachev-Mukai pairing (\cite[Theorem 4.7]{D}):
if ${\text{Eff}_{\mathbb{R}} Y}^{\wedge}
= \{D \in A^1(Y) \;|\; \langle D, \text{Eff}_{\mathbb{R}} \rangle_1 \geq 0\}$
then
\[
\text{Mov}(Y)=\text{Eff}_{\mathbb{R}} Y \cap {\text{Eff}_{\mathbb{R}} Y}^{\wedge}.
\]

For $Y^r_{r+3}$, the $(-1)$-Weyl hyperplanes
generate the Cox ring \cite[Theorem 1.2]{CT}.

We first introduce coordinates for the divisor $D\in \text{Pic}(Y)$, ie $D:=dH-\sum_{i=1}^{r+3} m_i E_i.$ We recall

\begin{theorem}\cite[Theorem 5.1]{bdp3}

{\textbf Case 1}. If the number of points $s=r+2$,
a divisor $D$ is effective if and only if

$$(A_{i}): d\geq m_i \text{ and } rd \geq \sum_{j=1, j\neq i}^{n+3} m_j$$

{\textbf Case 2}.  If the number of points $s=r+3$,
then a divisor $D$ is effective if and only if
inequalities $(A_{i})$ together with $(B_{n, I(t)})$ hold, where

$$(B_{n, I(t)}): \text{  }  k_{t, I(t)}:=[(t+1)n - t] d - t \sum_{i=1}^{r+3} m_i - \sum_{i\in I} m_i\geq 0$$
 
for every  set $I(t)$ so that $|I(t)|=r-2t+1$ where $-1\leq t \leq l+\alpha$ and $r=2l+\alpha$, $\alpha\in \{0, 1\}$.
\end{theorem}

We will recall  that in the Mori Dream Space cases,
the Weyl group and therefore the Weyl orbits
of a general line $h$ and a pencil of lines through one point $h-e_i$ are finite.

Denote the collection of all $(0)$- and $(1)$-curves on the Mori Dream Space $Y$ by
\begin{equation}\label{equation1}
\mathcal{E}:=\{\text{(i)-Weyl line}  \;\text{ for }\; i \in \{0,1\}\} \subset A^{r-1}(Y).
\end{equation}

Introduce $\mathcal{E}_{\geq 0}$ to be the collection of divisors
that intersect every $(0)$-curve and $(1)$-curve positively:
\[
\mathcal{E}_{\geq 0} :=
\{ D \in \text{Pic}(Y)_{\bbR} \text{ so that }  (D \cdot C) \geq 0
\text{ for every } C \in \mathcal{E} \}
\subset \text{Pic}(Y)_{\bbR}.
\]
Consider also the boundary,
\[
\mathcal{E}^{\partial} :=
\{ D\in \mathcal{E}_{\geq 0} \text{ so that }  (D \cdot C)= 0
\text{ for some } C \in \mathcal{E} \}
\subset \text{Pic}(Y)_{\bbR}.
\]

We will now prove the main result of this section.

\begin{theorem}\label{effective cone}
 If $s \leq r+3$ then
 $$\text{Eff}_{\mathbb{R}} Y^r_s = \mathcal{E}_{\geq 0}.$$
In other words, $(0)$- and $(1)$-Weyl lines give the extremal rays
for the cone of movable curves in $\mathbb{P}^r$ with $r+3$ points blown up.
\end{theorem}

\begin{proof}
We will first compute the set $\mathcal{E} \subset A^{r-1}(Y)$ of Equation \eqref{equation1}.

It is easy to see that $(0)$-Weyl lines consists only in a pencil of lines through one point and the rational normal curves of degree $r$ passing through $r+2$ general points.
\[
\alpha_{i}=h-e_i \text{  and  } \alpha'_{i}=nh - \sum_{j\neq i} e_j.
\]
The $(0)$-Weyl lines $\alpha_{i}$ give the facets $(A_{i})$.

We will leave to the reader to check that the $(1)$-curves are of the form
\[
\beta_{t, I(t)} = (t+1)r-t - \sum_{i\in I(t)} (t+1)e_i -  \sum_{i\notin I(t)} t e_i
\]
for every $-1\leq t\leq l+\alpha$ and $|I(t)|=r-2t+1$.
Indeed, ordering the multiplicities in a decreasing order,
one can see that performing a standard Cremona transformation to the last $r+1$ points, then
\[
Cr(\beta_{t, I(t)}) = \beta_{t+1, I(t+1)}.
\]

The $(1)$-Weyl line of minimal degree correspond to $t=-1$,
i.e. the general hyperplane class $\beta_{-1, I(-1)}=h$,
while $(1)$-Weyl line of maximal degree correspond to $t$ a half of $r$
i.e. when
for $\alpha=0$ ($r$ is even) is the quasihomogeneous curve $\beta_{l, I(l)}$,
and if $\alpha=1$, a homogeneous curve $\beta_{l+1, I(l+1)}$.
The $(1)$-Weyl line $\beta_{t, I(t)}$ with $|I(t)|=r-2t+1$ give facets $(B_{r,I(t)})$.
We obtain
\begin{equation}
\begin{split}
 \mathcal{E}&=\{ \{\alpha_i, \alpha'_i\}_{1\leq i\leq r+3}, \text{ and } \{\beta_{t, I(t)}\}_{ -1 \leq t \leq l} \},\\
\mathcal{E}_{\geq 0}&=\{D\in \text{Pic}(Y)_{\bbR} \text{ so that } D \cdot C\geq 0\ \text{for } C\in \mathcal{E}\},\\
&=\{D\in \text{Pic}(Y)_{\bbR} \text{ satisfying } (A_{i})\ \text{ and } (B_{n,I(t)})\},\\
&= \text{Eff}_{\mathbb{R}} Y.
\end{split}
\end{equation}
\end{proof}

\begin{remark}
For the Mori Dream Spaces $Y^3_7$
we can find $4$ types of curves in odd degree up to $7$ as $(1)$-Weyl lines
and $8$ types of curves in odd degree up to $15$ as $(0)$-Weyl lines.
For $Y^4_8$ there are seven types of $(0)$-Weyl lines up to degree $19$.
The Cox ring of $Y^3_7$ is generated by the $(-1)$-Weyl divisors
together with the anticanonical divisor,
and for $Y^4_8$ the effective cone was considered in \cite{Cas2}.
\end{remark}

It is natural to ask the following question:
\begin{question}\label{nms3}
Does the collection of $(0)$- and $(1)$-Weyl lines form all the extremal rays for $Y^3_7$ and $Y^4_8$?
\end{question}

\begin{corollary}
The faces of the effective cone of divisors on $\mathbb{P}^r$ with $r+3$ points are the components of $\mathcal{E}^{\partial}$.
\end{corollary}


If $Y$ is not a Mori Dream Space,
we define the infinite collection of curves
\begin{equation}
\begin{split}
\mathcal{I}:&=\{\text{(0)-Weyl lines and (1)-Weyl lines} \} \subset A^{r-1}(Y)\\
\mathcal{I}_{\geq 0}:&=\{D\in \text{Pic}(Y)_{\bbR} \text{ so that } D \cdot C\geq 0\ \text{for } C\in \mathcal{I}\} \subset \text{Pic(Y)}.
\end{split}
\end{equation}
i.e., $\mathcal{I}_{\geq 0}$ consists of
all divisors that intersect every curve of $\mathcal{I}$ positively.
The next theorem gives an infinite set of necessary conditions for the effectivity of a divisor on $Y^r_s$.

\begin{theorem}\label{nms} Let $Y=Y^r_{s}$; then
 $ \text{Eff}_{\mathbb{R}} Y \subset \mathcal{I}_{\geq 0}.$
\end{theorem}

\begin{proof}
We use that if $C$ is an $(i)$-Weyl line ($i=0,1)$, $q$ is a general point in $Y$
then there is another $(i)$-Weyl line through $q$ with the same class: these $(i)$-curves move in an appropriately large family.

Assume by contradiction that an $(i)$-Weyl line has negative intersection with a divisor $D$;
then choosing a point $q$ not in $D$
we find an $(i)$-Weyl line with the same class (and hence also negatively meeting $D$)
through $q$, which is a contradiction.
\end{proof}

\subsection{Base locus of Effective Divisors.}\label{baselocus}

We recall that the Picard group of $Y$, $\text{Pic}(Y)=<H, E_1,\cdots, E_s>$
where $H$ is the general hyperplane class and $E_i$ are the exceptional divisors.
We recall Definition \ref{def}:
a $(-1)-$ Weyl hyperplane is the Weyl orbit of the exceptional divisor.

Let $G$ be a  $(-1)$Weyl hyperplane;
then there exist a Weyl group element $w$ with $G=w(E_1)$. 
We recall that $\langle -, -\rangle_1$
is a bilinear form on the Picard group of $Y$ invariant under the Weyl group action.
In particular,
\[
\langle  G, G\rangle_1 = \langle w(E_1), w(E_1) \rangle_1= \langle E_1, E_1\rangle_1= - 1.
\]

\begin{theorem}\label{base locus 1}
If a $(-1)$-Weyl line $C$ and a $(-1)$-Weyl hyperplane $G$
are part of the base locus the linear system $|D|$ of an effective divisor $D$,
then
$$(C \cdot G)=0.$$
\end{theorem}

\begin{proof}
Assume by contradiction that there exist an effective divisor $D$
containing in the base locus the Weyl hyperplane $G$ and the $(-1)$ Weyl curve $C$,
with $(G \cdot C) \geq 1$.
Let $G=\sigma(E_1)$ for $\sigma \in \text{Weyl}(Y)$,
and note that $\langle G, G \rangle_1=-1$.
By \cite[Lemma 7.1]{DP} we have for some $p>0$,
\[
-p= \langle D, G\rangle_1 = \langle D, \sigma(E_1)\rangle_1 = \langle \sigma^{-1}(D), E_1 \rangle_1.
\]
The bilinearity this implies that
$$\langle D-pG, G \rangle_1 = \langle D, G\rangle_1 -p \langle G, G\rangle_1=0.$$

We know that $D-pG$ is an effective divisor,
and therefore $\sigma^{-1}(D-pG)$ is also an effective divisor. 
Moreover,
$$ \langle \sigma^{-1}(D - p G), E_1\rangle_1 = \langle D - pG, \sigma(E_1)\rangle_1 =\langle  D - p G, G\rangle_1=0.
$$
Therefore the divisor $\sigma^{-1}(D - p G)$ is based at at most $s-1$ points,
missing the point $E_1$.

We assume that 
$$
1\leq (D \cdot C) = (\sigma(E_1) \cdot C) = (E_1 \cdot \sigma^{-1}(C)).
$$
The positive intersection $(E_1 \cdot \sigma^{-1}(C))$
implies that $\sigma^{-1}(C)$ is an effective curve
(i.e. not contracted by the Weyl group element $\sigma$).
Therefore $\sigma^{-1}(C)$ is a $(-1)$ Weyl line that
passes through the point $E_1$.
Moreover, since $p>0$ we have
$$
(\sigma^{-1}(D - p G) \cdot \sigma^{-1}(C)) = ((D - p G) \cdot C)
= (D\cdot C) - p (G \cdot C) <0.
$$
We conclude that $\sigma^{-1}(C)$
is the base locus of the effective divisor $\sigma^{-1}(D - p G)$.
This is a contradiction since $\sigma^{-1}(C)$ passes through the point $E_1$
while the divisor $\sigma^{-1}(D - p G)$ doesn't pass through the point $E_1$.
Therefore $\sigma^{-1}(C)$ is a family of curves 
sweeping out the effective divisor $\sigma^{-1}(D - p G)$ and this is a contradiction.
\end{proof}

\begin{theorem}\label{weyl div}
Let  $G_1$ and $G_2$ be two $(-1)$-Weyl hyperplanes in the base locus of an effective divisor $D$. Then
$$\langle  G_1, G_2\rangle_1 =0$$
\end{theorem}

\begin{proof}
By \cite[Lemma 7.1]{DP} we know

$$\langle  D, G_1\rangle_1 <0$$
$$\langle D, G_2\rangle_1 <0$$
and we claim that $G_1$ and $G_2$ are orthogonal with respect to the Dolgachev Mukai pairing.

Let $\sigma$ an element of the Weyl group so that $G_1=\sigma(E_1)$.
Then for some $p>0$ (the multiplicity of containtment of $G_1$ in $D$)
\[
-p= \langle D, G_1\rangle_1 = \langle D, \sigma(E_1)\rangle_1 =\langle \sigma^{-1}(D), E_1\rangle_1.
\]
Therefore,
$$
\langle D-pG_1, G_1\rangle_1 =\langle D, G_1\rangle_1 -p \langle G_1, G_1\rangle_1 =0.
$$
 Moreover,
$$
\langle \sigma^{-1}(D - p G_1), E_1\rangle_1 = \langle  D - pG_1, \sigma(E_1)\rangle_1 = \langle D - p G_1, E_1\rangle_1=0.
$$
Therefore the effective divisor $\sigma^{-1}(D - p G_1)$ is based at at most $s-1$ points,
missing the point $E_1$.
Assume by contradiction that 
$$
1\leq  \langle G_1, G_2\rangle_1 = \langle \sigma(E_1), G_2\rangle_1 =\langle E_1, \sigma^{-1}(G_2)\rangle_1.
$$
Since $p>0$ we obtain
$$ \langle  \sigma^{-1}(D - p G_1), \sigma^{-1}(G_2)\rangle_1 = \langle D - pG_1, G_2\rangle_1 =\langle D, G_2\rangle_1 - p\langle G_1, G_2\rangle_1 <0.
$$

By \cite[Lemma 7.1]{DP},
we have that the $(-1)$ Weyl divisor $\sigma^{-1}(G_2)$
splits off the divisor $\sigma^{-1}(D-pG_1)$.
This gives a contradiction since the divisor $\sigma^{-1}(G_2)$
passes through the point $E_1$ 
and the divisor $\sigma^{-1}(D-pG_1)$ does not pass through the point $E_1$.

\end{proof}

Let $D$ be an effective divisor on $Y$.
Necessary conditions for effectivity are given in Section \eqref{cone}.
We remark that Theorem \eqref{weyl div}
does not hold for curves contained in an effective divisor $D$,
with respect to the bilinear form $\langle -,-\rangle$.

\begin{example}\label{examp}
Consider the effective divisor $D:=6H -\sum_{i=1}^9 4E_i$ in $\mathbb{P}^5$.
We can see that two types of $(-1)$ Weyl lines are contained in $D$:
one is $C$ (the rational normal curve of degree $5$ passing through first $8$ points)
and the other is $L_{19}$ (the line through points $1$ and $9$)
are contained in the base locus of $D$,
however $\langle C, L_{19} \rangle_1=5-4=1>0.$
\end{example}

Moreover, \cite[Corollary 8.3]{DM1} shows that
if two Weyl surfaces (ie the Weyl orbit of a plane through three fixed points)
can not be fixed part of a divisor in $\mathbb{P}^4$ based at $8$ points then,
their intersection in the Chow ring is zero.

This makes us predict that for an effective divisor $D$ on $Y$,\
it has a resolution of singularities by blowing up its base locus for $D$,
with the blown up space $\tilde{Y}$ being smooth.
In the next subsection we show that $(-1)$
curves create a contribution to the dimension of the linear system of an effective divisor $D$.

\end{document}